\documentclass[11pt,twoside]{article}
\usepackage{amsfonts}
\usepackage{amsmath,amssymb}
\usepackage{amsthm}
\usepackage[mathscr]{euscript}
\usepackage{mathrsfs}
\usepackage{indentfirst}
\usepackage{color}\usepackage{enumitem}

\newtheorem{theorem}{Theorem}[section]
\newtheorem{lemma}[theorem]{Lemma}
\newtheorem{corollary}[theorem]{Corollary}

\newtheorem{remark}[theorem]{Remark}
\newtheorem{proposition}[theorem]{Proposition}
\numberwithin{equation}{section}

\newcommand{\lbl}[1]{\label{#1}}
\allowdisplaybreaks

\newcommand{\be}{\begin{equation}}
\newcommand{\ee}{\end{equation}}
\newcommand\bes{\begin{eqnarray}} \newcommand\ees{\end{eqnarray}}
\newcommand{\bess}{\begin{eqnarray*}}
\newcommand{\eess}{\end{eqnarray*}}
\newcommand{\bbbb}{\left\{\begin{aligned}}
\newcommand{\nnnn}{\end{aligned}\right.}
\newcommand{\bea}{\begin{align*}}
\newcommand{\eea}{\end{align*}}

\newcommand\ep{\varepsilon}

\newcommand\kk{\left}
\newcommand\rr{\right}
\newcommand\dd{\displaystyle}

\newcommand\df{\dd\frac}

\newcommand\yy{\infty}

\newcommand\R{\mathbb{R}}

\newcommand\ol{\overline}

\newcommand\sk{\smallskip}

\begin{document}\thispagestyle{empty}
\setlength{\baselineskip}{16pt}

\begin{center}{\LARGE\bf A viral propagation model with nonlinear}\\[2mm]
 {\LARGE\bf infection rate and free boundaries\footnote{This work was supported by NSFC Grants 11771110, 11971128}}\\[4mm]
  {\Large  Lei Li$^a$, \ Siyu Liu$^{b,c}$, \ Mingxin Wang$^{a,}$\footnote{Corresponding author. {\sl
E-mail}: mxwang@hit.edu.cn}}\\[1.5mm]
{a. School of Mathematics, Harbin Institute of Technology, Harbin 150001, PR China.\\[1.5mm]
b. School of Science and Engineering, The Chinese University of Hong Kong, Shenzhen  518172, China.\\[1.5mm]
c. School of Mathematical Sciences, University of Science and Technology of China, Hefei  230026, China.	
}
\end{center}

\date{\today}

\begin{abstract} In this paper we put forward a viral propagation model with nonlinear  infection rate and free boundaries and investigate the dynamical properties. This model is composed of two ordinary differential equations and one partial differential equation, in which the spatial range of the first equation is the whole space $\mathbb{R}$, and the last two equations have free boundaries. As a new mathematical model, we prove the existence, uniqueness and uniform estimates of global solution, and provide the criteria for spreading and vanishing, and long time behavior of the solution components $u,v,w$. Comparing with the corresponding ordinary differential systems, the {\it Basic Reproduction Number} ${\cal R}_0$ plays a different role. We find that when ${\cal R}_0\le 1$, the virus cannot spread successfully; when ${\cal R}_0>1$, the successful spread of virus depends on the initial value and varying parameters.


\textbf{AMS Subject Classification (2100)}: 35K57, 35B40, 35R35, 92D30

\end{abstract}

\section{Introduction}
\underline{Background}
\
In order to clarify the pathogenesis of diseases and seek
effective treatment measures, viral dynamics have been a hot research topic
(cf. \cite{Wei,Per}), which usually cannot be answered by biological
experimental methods alone but require the help of mathematical models. For this
reason, a simple model was introduced few decades ago by Nowak and Bangham
\cite{Noba}. See also Nowak and May \cite{NoM}. The basic model of viral dynamics is the
following set of differential equations
\bes\begin{cases}
	u'= \theta-au-buw,\\
	v'=buw-cv,\\
	w'=kv-qw,
\end{cases}\lbl{x}\ees
where $u$, $v$ and $w$ represent the population of uninfected cells, infected cells
and viruses, respectively; uninfected cells
are produced at a constant rate $\theta$ and with death rate $au$; $cv$
is the death rate of infected cells; virus
particles $w$ infect uninfected cells with rate $buw$, and meanwhile virus particles
are produced by infected cells with rate $kv$ and have death rate $qw$. It has been shown
that if the {\it Basic Reproduction Number} ${\cal R}_0={\theta kb}/(acq)<1$, then the
system returns to the uninfected state $(\theta/a,0,0)$. If ${\cal R}_0>1$, then the
system will converge to the unique positive equilibrium state $\big(\frac{qc}{kb}, \, \frac{\theta}{c}-\frac{aq}{kb}, \, \frac{\theta
	k}{qc}-\frac{a}{b}\big)$. This indicates that in the initial stage of infection,
if each infected cell infects less than one cell on average, then the infection cannot
spread; if each infected cell infects ${\cal R}_0>1$ cells on average, the number of infected
cells increases and the number of uninfected cells declines.

\sk\underline{Mathematical Model} \ To investigate the impact of spatial dynamics on this model, Stancevic et al. \cite{Sta} extended this model to include spatially random diffusion and spatially directed chemotaxis. Invoked by their ideas, we give
the basic model assumptions as follows:

\sk(i) A nonlinear infection rate can happen due to saturation at high virus concentration, where the infectious fraction is so high that exposure is very likely. Moreover, with the increase of the virus concentration the living environment for cells becomes worse and worse. Thus, it is reasonable for us to assume that the rate of infection for virus and the virion production rate for infected cells are both nonlinear. Here we use
\[f_1(u,w)=\theta-au-\frac{buw}{1+w}, \ \ f_2(u,v,w)=\frac{buw}{1+w}-cv, \ \ f_3(v,w)=\frac{kv}{1+w}-qw\]
instead of the three terms in the right hand side of \eqref{x}.

\sk(ii) We assume that the major spatial dispersal comes from the moving (diffusion)
of viruses in vivo, while both the
uninfected and infected cells are immobile (do not diffuse). So we add only a
diffusion term to the differential equation of
viruses;

\sk(iii) Since the infected cells are caused by viruses, their distribution range is
the same;

\sk(iv) The distribution of viruses and infected cells is a local range, which is
small relative to the distribution of uninfected
cells, so we think that uninfected cells are distributed over the whole space.
Such kind of assumptions have been used in
the species invasion models (cf. \cite{DL2,ZW2014,ZhaoW16} for example);

\sk(v) Initially, viruses are distributed over a local range $\Omega_0$ (the
initial habitat). They will spread from boundary to
expand their habitat as a result of the spatial dispersal freely. That is, as time
$t$ increases, $\Omega_0$ will evolve
into expanding region $\Omega(t)$ with expanding front $\partial\Omega(t)$. Initial
function  $w_0(x)$, and as a result
$v_0(x)$, will evolve into positive functions $w(t,x)$ and $v(t,x)$ which vanish on
the moving boundary $\partial\Omega(t)$;

\sk(vi) For simplicity, we restrict our problem to the one dimensional case. Based on
the {\it deduction of free boundary
	conditions} given in \cite{5-hdu12}, we have the following free boundary conditions
$$g'(t)=-\mu w_x(t,g(t)), \ \ \ h'(t)=-\beta w_x(t,h(t)).$$

All of these assumptions (i)-(vi) suggest the following model, which governs
the spatial and temporal evolution of
viruses and cells, as well as free boundaries:
\be\begin{cases}
	u_t=f_1(u,w),  &t>0, \ -\infty<x<\infty,\\
	v_t=f_2(u,v,w), &t>0, \ g(t)<x<h(t),\\
	w_t-dw_{xx}=f_3(v,w), &t>0, \ g(t)<x<h(t),\\
	v(t,x)=w(t,x)=0, &t>0, \ x\notin(g(t),h(t)),\\
	g'(t)=-\mu w_x(t,g(t)),\,\,\,h'(t)=-\beta w_x(t,h(t)), &t\ge0,\\
	u(0,x)=u_0(x), &-\infty<x<\infty,\\
	v(0,x)=v_0(x), \ \ \ w(0,x)=w_0(x), &-h_0\le x\le h_0,\\
	h(0)=-g(0)=h_0,
	\label{1.1}
\end{cases}\ee
where $x=g(t)$ and $x=h(t)$ are the moving boundaries to be determined together
with $u(t,x)$, $v(t,x)$ and $w(t,x)$; $d,~\theta,~a,~b,~c,~k,~q,~\mu,~\beta,~h_0$ are positive constants.

Denote by $C^{1-}(I)$ the space of Lipschitz continuous functions in $I$. We
assume that the initial functions $u_0,v_0,w_0$
satisfy
\begin{equation}\label{1.2}
\left\{\begin{array}{ll}
u_0\in C^{1-}(\mathbb{R})\cap L^\yy(\mathbb{R}), \ \ v_0\in C^{1-}([-h_0,h_0]), \
\ w_0\in W^2_p((-h_0,h_0)),\\[.5mm]
v_{0}(\pm h_{0})=w_{0}(\pm h_{0})=0, \ \ w'(-h_0)>0, \ \ w'(h_0)<0,\\[.5mm]
u_{0}>0 \ \ \text{in }\,\mathbb{R},\ \ \ v_0,w_0>0 \ \ \text{in}\,(-h_{0},h_{0})
\end{array}\rr.
\end{equation}
with $p>3$. Denote by $L_0$ and $L_*$ the Lipschitz constant of $u_0$ and $v_0$,
respectively.

Partially degenerate reaction-diffusion systems, which mean that several diffusion
coefficients are zeros, have been increasingly applied to epidemiology, population biology etc; see \cite{VL,MG}, for example. Some researchers have introduced the Stefan type free boundary to the partially degenerate systems, please refer to \cite{Ahn, LHWdcdsb20, LW, WC-NA15} and the references therein.

\underline{Aims and Main Results} \ This paper concerns with the dynamics of
\eqref{1.1}. The global existence, uniqueness, regularity and uniform estimates in time $t$ of solution are first studied. Then a spreading-vanishing dichotomy is established, i.e., either

(i) \underline{{\it Spreading}} (virus persistence): the virus successfully
infects the uninfected cells and spreads itself to the
uninfected area in the sense that $\dd\lim_{t\to\yy}h(t)=-\lim_{t\to\yy}g(t)=\infty$, and
\bess
\dd\limsup_{t\to\yy}\|v(t,\cdot)\|_{C([g(t),h(t)])}>0, \
\dd\limsup_{t\to\yy}\|w(t,\cdot)\|_{C([g(t),h(t)])}>0.
\eess
In addition, if $\mathcal{R}_0+\sqrt{\mathcal{R}_0}>b/a$, where ${\cal R}_0={\theta kb}/(acq)$, then
\bess
\dd\underline{u}_\infty\le\liminf_{t\to\infty}u(t,x)\le\limsup_{t\to\infty}u(t,x)\le\bar{u}_\infty,\\
\dd\underline{v}_\infty\le\liminf_{t\to\infty}v(t,x)\le\limsup_{t\to\infty}v(t,x)\le\bar{v}_\infty,\\
\dd\underline{w}_\infty\le\liminf_{t\to\infty}w(t,x)\le\limsup_{t\to\infty}w(t,x)\le\bar{w}_\infty
\eess
locally uniformly in $\mathbb{R}$ for some positive constants $\underline{u}_\infty$, $\bar{u}_\infty$, $\underline{v}_\infty$, $\bar{v}_\infty$, $\underline{w}_\infty$ and $\bar{w}_\infty$. Particularly, under a stronger assumption that $b\le 2a$, we will derive
\[\lim_{t\to\yy}u(t,x)=u^*,\ \ \lim_{t\to\yy}v(t,x)=v^*,\ \ \lim_{t\to\yy}w(t,x)=w^* \ \ {\rm locally~uniformly~in~}\mathbb{R},\]
where $(u^*,v^*,w^*)$ is the unique positive root of \eqref{3.10};

or

(ii) \underline{{\it Vanishing}} (virus dies out): the virus $w$ and the
infected cells $v$ will vanish in a bounded area, i.e.,
$-\yy<\dd\lim_{t\to\yy}g(t)<\lim_{t\to\yy}h(t)<\yy$ and
$$	\dd\lim_{t\to\infty}\|v(t,\cdot)\|_{C([g(t),\,h(t)])}=
\lim_{t\to\infty}\|w(t,\cdot)\|_{C([g(t),\,h(t)])}=0, \ \
\lim_{t\to\infty}u=\theta/{a} \ \ \mbox{uniformly\, in } \ \mathbb{R}.$$
Moreover, $\dd\lim_{t\to\yy}h(t)-\lim_{t\to\yy}g(t)\le \pi\sqrt{acd/(\theta kb-acq)}$ if ${\cal R}_0>1$.

\sk As for the {\it Basic Reproduction Number} ${\cal R}_0=\theta kb/(acq)$, in our results, we will show that it plays a different role, comparing
with the corresponding ordinary differential systems.  When ${\cal R}_0\le 1$, {\it
	vanishing} always happens, that is, the virus
cannot spread successfully. On the other hand, when ${\cal R}_0>1$, we have a
criterion as follows: if the initial occupying
area $[-h_0,h_0]$ is beyond a critical size, namely $2h_0\ge
\pi\sqrt{acd/(\theta kb-acq)}$, then {\it spreading} happens
regardless of the moving parameter $\mu$, $\beta$ and initial population density
$(u_0,v_0,w_0)$. While $2h_0<\pi\sqrt{acd/(\theta kb-acq)}$, whether {\it spreading} or {\it vanishing} happens
depends on the initial population density
$(v_0,w_0)$ and the moving parameter $\mu$ and $\beta$.

The paper is organized as follows. Section 2 concerns with the existence,
uniqueness and uniform estimates of global solution. In Section 3 we give some preliminaries which will be used later. In Section 4 we study the long time behavior of solution components $u,v,w$,
and in Section 5 we discuss the criteria for {\it spreading} and {\it vanishing}.
At the last section, we give a brief discussion.

Before ending this section we mention that in recent years, more and more free
boundary problems of reaction diffusion systems have been introduced to describe the dynamics of species after the pioneering work \cite{DL10}. Interested readers
can refer to, except for the above cited papers, \cite{DWZ15, GW, GW15, WZjdde14, WZhang17} for competition models, \cite{WWmma18,Wjde14,WZjde18} for prey-predator models.

\section{Existence, uniqueness and uniform estimates of solution of (\ref{1.1})}

In this section we prove the global existence and
uniqueness of the solution to problem (\ref{1.1}).
For convenience, we first introduce some notations.
Denote
\begin{center}
	$A_1=\max\left\{\|u_{0}\|_\infty,\,\theta/a\right\},\ \ B_1=\|v_{0}\|_\infty+1,
	\ \ B_2=\|w_{0}\|_\infty+1,$\\[1mm]
	$\mathcal{A}=\{a,b,c,d,k,q,h_0,\mu,\beta,\alpha,
	A_1,B_1,B_2,\|w_0\|_{W^2_p((-h_0,h_0))},w'_0(\pm h_0)\},$\\[1mm]
	$\Pi_T=[0,T]\times\mathbb{R},\,\ \ \Delta_T=[0,T]\times [-1,1],\,\ \
	D^{T}_{g,h}=\left\{0\le t\le T,\ g(t)<x<h(t)\right\}$.
\end{center}
Let $X$ be a Banach space and $\varphi,\psi\in X$. Denote
$\|\varphi,\,\psi\|_X=\max\{\|\varphi\|_X,\,\|\psi\|_X\}$  for
simplicity.

\begin{theorem}{\rm(Local solution)}\label{th2.1}\, For any given $\alpha\in(0,1)$
	and $p>3/(1-\alpha)$, there exists a $T>0$
	such that the problem \eqref{1.1} has a unique local solution $(u,v,w,g,h)\in
	C^{1,1-}(\Pi_{T})\times C^{1,1-}(\ol
	D^T_{g,h})\times W^{1,2}_p(D^T_{g,h})\times [C^{1+\frac{\alpha}2}([0,T])]^2$.
	Moreover,
	\bess
	u>0\ \ {\rm in} \ \ \Pi_T;\ \ v,w>0\ \ {\rm in}\ \ D^T_{g,h};\ \
	g'(t)<0,\ \ h'(t)>0\ \ {\rm in}\ \ [0, T],
	\eess
	where $u\in C^{1,1-}(\overline D^T_{g,h})$ means that $u$ is continuously differentiable
	in $t\in[0,T]$ and is Lipschitz
	continuous in $x\in[g(t),h(t)]$ for all $t\in[0,T]$.
\end{theorem}

\begin{proof} Invoked by the proof of \cite[Theorem 2.1]{LHWdcdsb20} and \cite[Theorem 1.1]{LW}, we
	divide the proof into several steps. Unless otherwise specified in the proof, positive constants $C_i$ depend only on $\mathcal{A}$.
	
\sk{\it Step 1:} Given $T>0$, we say $u\in C^{1-}_{x}(\Pi_T)$ if there is a constant
	$L_u(T)$ such that
	\[|u(t,x_1)-u(t,x_2)|\le L_u(T)|x_1-x_2|, \ \ \forall \ x_1,x_2\in\mathbb{R},\,\,
	0<t\le T.\]
	For $s>0$, define
	\[\mathbb X_{u_0}^s=\left\{\phi\in C(\Pi_s): \phi(0,x)=u_0(x),\ 0\le\phi\le
	A_1\right\}.\]
	For any given $u\in\mathbb X_{u_0}^1\cap C^{1-}_{x}(\Pi_1)$ we consider the
	following problem
	\bes
	\begin{cases}
		v_t=f_2(u(t,x),v,w), &t>0, \ g(t)<x<h(t),\\
		w_t-dw_{xx}=f_3(v,w), &t>0, \ g(t)<x<h(t),\\
		v(t,x)=w(t,x)=0, &t>0, \ x\notin(g(t),h(t)),\\
		g'(t)=-\mu w_x(t,g(t)),\, h'(t)=-\beta w_x(t,h(t))&t\ge0,\\
		v(0,x)=v_0(x), w(0,x)=w_0(x), &|x|\le h_0,\\
		h(0)=-g(0)=h_0.
	\end{cases}
	\label{2.1}
	\ees
	By \cite[Theorem 1.1]{LW}, we know that for some $0<T\ll1$, \eqref{2.1} has a
	unique solution $(v,w,g,h)\in C^{1,1-}(\ol
	D^T_{g,h})\times C^{\frac{1+\alpha}{2},1+\alpha}(\ol D^T_{g,h})\times
	[C^{1+\frac{\alpha}2}([0,T])]^2$. Moreover,
	\bes\label{2.2}
	\begin{cases}
		-w_x(t,h(t)),\ w_x(t,g(t))>0 \ \ \mbox{in} \ \ [0,T];\ \ 0<v\le B_1,\ 0<w\le B_2 \
		\ \mbox{in} \ \ D^T_{g,h},\\
		\|w\|_{W^{1,2}_p(D^T_{g,h})}+\|w_x\|_{C(\ol D^T_{g,h})}
		+\|g,h\|_{C^{1+\frac{\alpha}2}([0,T])}\le M,
	\end{cases}
	\ees
	where $M$ depends only on $\mathcal{A}$.
	
\sk{\it Step 2:} For the function $w(t,x)$ obtained in Step 1, we consider the
	following parameterized ODE problem,
	\bes\left\{\begin{aligned}
		&\tilde u_t=f_1(\tilde u,w(t,x)), && (t,x)\in (0,T]\times\mathbb{R},\\
		&\tilde u(0;x)=u_0(x)>0, && x\in\mathbb{R}.
	\end{aligned}\right.\label{2.3}\ees
	By the standard ODE theory, \eqref{2.3} has a unique solution $\tilde u\in
	C^{1,1-}(\Pi_T)$ and $0<\tilde u\le A_1$.
	
	Now we estimate the Lipschitz constant of $\tilde{u}$ in $x$.
	Since it can be easily derived from \eqref{2.2} that $|w(t,x_1)-w(t,x_2)|\le M|x_1-x_2|$ for any given $(t,x_1),\ (t,x_2)\in \Pi_T$, we have
	\bess |\tilde u(t,x_1)-\tilde u(t,x_2)|&=&\left|\int_{0}^{t}\tilde
	u_t(s,x_1)-\tilde u_t(s,x_2){\rm
		d}s+u_0(x_1)-u_0(x_2)\right|\\
	&\le& \int_{0}^{t}|\tilde u_t(s,x_1)-\tilde u_t(s,x_2)|{\rm d}s+L_0|x_1-x_2|\\
	&\le& L_0|x_1-x_2|+\int_{0}^{t}|f_1(\tilde u(s,x_1),w(s,x_1))-f_1(\tilde u(s,x_2),w(s,x_2))|{\rm d}s\\
	&\le& \int_{0}^{t}(a+b)|\tilde u(s,x_1)-\tilde u(s,x_2)|{\rm
		d}s+(bA_1TM+L_0)|x_1-x_2|.
	\eess
	Then noticing $0<T\le1$ and making use of the Gronwall inequality, we obtain
	\bess |\tilde u(t,x_1)-\tilde u(t,x_2)| \le(bA_1M+L_0)e^{a+b}|x_1-x_2|.
	\eess
	This shows that $L_{\tilde{u}}=(bA_1M+L_0)e^{a+b}$ is the Lipschitz constant of
	$\tilde{u}$. Define
	\[\mathbb Y_{u_0}^T=\left\{\phi\in C(\Pi_T): \phi(0,x)=u_0(x),\ 0\le\phi\le A_1,\
	|\phi(t,x)-\phi(t,y)|\le L_{\tilde{u}}|x-y|\right\}.\]
	Obviously, $\mathbb Y_{u_0}^T$ is complete with the metric
	$d(\phi_1,\phi_2)=\sup_{(t,x)\in\Pi_{T}}|\phi_1(t,x)-\phi_2(t,x)|$. The above
	analysis allows us to define a map $\mathcal{F}(u)=\tilde u$, and $\mathcal{F}$
	maps $\mathbb Y_{u_0}^T$ into itself.
	
\sk{\it Step 3:} We are in the position to prove that $\mathcal{F}$ is a contraction
	mapping in $Y_{u_0}^T$ for sufficiently small $T$. In fact, for $i=1,2$, let $v_i,w_i,g_i,h_i$ be the unique solution of \eqref{2.1} with $u=u_i$. By arguing as in the proof of \cite[Theorem 1.1]{LW}, we can show
	that there exists a constant $L_v$, which only depends on $ \mathcal{A}$, such that for any given $(t,x_1),(t,x_2)\in \ol D^T_{g_i,h_i}$,
	\bes\label{2.4}
	|v_i(t,x_1)-v_i(t,x_2)| \le L_v|x_1-x_2|.
	\ees
	
	Denote $U=u_1-u_2$, $\tilde U=\tilde u_1-\tilde u_2$, $V=v_1-v_2$ and $W=w_1-w_2$. Since $\tilde u_i$ satisfy
	\bess\left\{\begin{aligned}
		&\tilde u_{i,t}=f_1(\tilde u_i,w_i), && (t,x)\in
		(0,T]\times\mathbb{R},\\
		&\tilde u_i(0,x)=u_0(x)>0,  && x\in\mathbb{R},
	\end{aligned}\right.\eess
	it follows that, for any $(t,x)\in \Pi_T$,
	\bess|\tilde U(t,x)|\le \int_{0}^{t}(a+b)|\tilde U(s,x)|{\rm
		d}s+bA_1T\|W\|_{L^{\infty}(\Pi_T)}.\eess
	By virtue of the Gronwall inequality again, it yields
	\bes|\tilde U(t,x)|\le bA_1e^{a+b}T\|W\|_{L^{\infty}(\Pi_T)}.\label{2.5}\ees
	
	The following arguments are devoted to an estimate of $\|W\|_{L^{\infty}(\Pi_T)}$.
	Evidently, $w_i$ satisfy
	\bess
	\left\{\begin{aligned}
		&w_{i,t}-dw_{i,xx}=f_3(v_i,w_i),& &0<t\le T,~g_i(t)<x<h_i(t),\\
		&w_i(t,x)=0,& &0<t\le T,~x\notin(g_i(t),h_i(t)),\\
		&w_i(0,x)=w_0(x),&&|x|\le h_0.
	\end{aligned}\right.
	\eess
	We straighten the boundaries and define
	\[x_i(t,y)=\frac{(h_i(t)-g_i(t))y+h_i(t)+g_i(t)}2,\ \ z_i(t,y)=w_i(t,x_i(t,y)), \ \ r_i(t,y)=v_i(t,x_i(t,y)).\]
	For simplicity, we introduce the following notations
	$\xi=\xi_1-\xi_2$, $\zeta=\zeta_1-\zeta_2$, $z=z_1-z_2$, $r=r_1-r_2$,
	$h=h_1-h_2$, $g=g_1-g_2$, where
	\bess \xi_i(t)=\df{4}{(h_i(t)-g_i(t))^2},\ \
	\zeta_i(t,y)=\df{h'_i(t)+g'_i(t)}{h_i(t)-g_i(t)}+\df{(h'_i(t)-g'_i(t))y}{h_i(t)-g_i(t)}.
	\eess
	Then $z$ satisfies
	\bess
	\left\{\begin{aligned}
		&z_t-d\xi_1z_{yy}-\zeta_1z_{y}
		=d\xi z_{2,yy}+\zeta z_{2,y}+\frac{kr}{1+z_1}-\frac{kr_2z}{(1+z_1)(1+z_2)}-qz, &&0<t\le T,~|y|<1,\\
		&z(t,\pm 1)=0, &&0\le t\le T,\\
		&z(0,y)=0, &&|y|\le 1.
	\end{aligned}\right.
	\eess
	By the $L^p$ estimates for parabolic equations, we see
	\[\|z\|_{W^{1,2}_p(\Delta_T)}\le
	C_1\big(\|g,\,h\|_{C^1([0,T])}+\|r\|_{C(\Delta_T)}\big).\]
	
	We now estimate $\|r\|_{C(\Delta_T)}$. For any given $(t,y)\in \Delta_T$, it
	follows that
	\bess |r(t,y)|=|v_1(t,x_1(t,y))-v_2(t,x_2(t,y))|
	\le |V(t,x_1(t,y))|+|v_2(t,x_1(t,y))-v_2(t,x_2(t,y))|.
	\eess
	It follows from the inequality \eqref{2.4} that
	\[|v_2(t,x_1(t,y))-v_2(t,x_2(t,y))|\le C_2\|g,h\|_{C^1([0,T])}.\]
	Additionally, we can prove the following inequality:
	\bes
	|V(t,x_1(t,y))|\le
	C_3\kk(\|g,h\|_{C^1([0,T])}+T\|U,\,W\|_{L^{\infty}(\Pi_T)}\rr).
	\label{x1}\ees
	Its proof will be put in the next step on account of the length. Thus we have
	\bess\|r\|_{C(\Delta_T)}\le
	C_4\left(\|g,\,h\|_{C^1([0,T])}+T\|U,\,W\|_{L^{\infty}(\Pi_T)}\right).\eess
	Then it follows that
	\[\|z\|_{W^{1,2}_p(\Delta_T)}\le
	C_5\left(\|g,\,h\|_{C^1([0,T])}+T\|U,\,W\|_{L^{\infty}(\Pi_T)}\right).\]
	By utilizing the similar methods in Step 2 of \cite[Theorem 2.1]{LSWjmaa20} and
	the embedding theorem: $[z_y]_{C^{\frac\alpha 2,\alpha}(\Delta_T)}\le C\|z\|_{W^{1,2}_p(\Delta_T)}$ for
	some positive constant $C$ independent of $T^{-1}$ (\cite[Theorem 1.1]{Wdcds19}), we can show that
	\[\|W\|_{L^{\infty}(\Pi_T)}\le
	C_6\left(T\|W\|_{L^{\infty}(\Pi_T)}+\|U\|_{L^{\infty}(\Pi_T)}\right).\]
	Hence
	\[\|W\|_{L^{\infty}(\Pi_T)}\le 2C_6\|U\|_{L^{\infty}(\Pi_T)} \ \ \
	\mbox{if} \ \ 0<T\ll 1.\]
	This combined with \eqref{2.5} arrives at
	\[\|\tilde U\|_{L^{\infty}(\Pi_T)}\le
	C_7T\|U\|_{L^{\infty}(\Pi_T)}\le \df{1}{2}\|U\|_{L^{\infty}(\Pi_T)} \
	\ \ \mbox{if} \ \ 0<T\ll 1.\]
	As a consequence, $\mathcal{F}$ is a contraction mapping and there exists a
	unique local solution $(u,v,w,g,h)$. Moreover
	the desired properties of the local solution can be obtained from the above
	arguments.
	
\sk{\it Step 4:} In this step, we are going to tackle the estimate \eqref{x1}, which
	will be divided into several cases. By the
	definition of $x_1(t,y)$, it is easy to see that $g_1(t)\le x_1(t,y)\le h_1(t)$. We
	denote $x_1=x_1(t,y)$ for simplicity.
	
\sk{\it Case 1:} $x_1\notin (g_2(t),h_2(t))$. In this case $v_2(t,x_1)=0$, and
	either $g_1(t)\le x_1\le g_2(t)$ or $h_2(t)\le x_1\le h_1(t)$. We only deal with the former case. Hence
	\begin{align*}
	|V(t,x_1)|&=|v_1(t,x_1)-v_1(t,g_1(t))|\le L_v|x_1-g_1(t)|\\
	&\le L_v|g_2(t)-g_1(t)|\le L_v\|g\|_{C^1([0,T])}.
	\end{align*}
	
	\sk{\it Case 2:} $x_1\in (g_2(t),h_2(t))$ and either $x_1>h_0$ or $x_1<-h_0$. We deal with only the case $x_1>h_0$. Then we can uniquely find
	$0<t_{x_1}, t'_{x_1}\le t$ such that $h_1(t_{x_1})=x_1$ and $h_2(t'_{x_1})=x_1$. Then
	$v_1(t_{x_1},x_1)=v_2(t'_{x_1},x_1)=0$. Without loss of generality, we assume
	$t'_{x_1}>t_{x_1}$. Then $h_1(t_{x_1}')>h_1(t_{x_1})=x_1=h_2(t_{x_1}')$, $x_1\in(g_1(s), h_1(s))\cap(g_2(s),h_2(s))$ for all $t_{x_1}'<s\le t$ and $x_1\in(g_1(t_{x_1}'), h_1(t_{x_1}'))\setminus(g_2(t_{x_1}'), h_2(t_{x_1}'))$. Hence,
	\[|V(t_{x_1}',x_1)|=v_1(t_{x_1}',x_1)\le L_v \|g,\,h\|_{C([0,T])}\]
	by the conclusion of Case 1. Integrating the differential equation of $v_i$ from $t_{x_1}'$ to $s \,(t_{x_1}'<s\le t)$ we obtain
	\bess
	v_1(s,x_1)&=&v_1(t_{x_1}',x_1)+\int_{t_{x_1}'}^s f_2(u_1, v_1,w_1)\big|_{x=x_1}{\rm d}\tau,\\
	v_2(s,x_1)&=&\int_{t_{x_1}'}^sf_2(u_2, v_2,w_2)\big|_{x=x_1}{\rm d}\tau.
	\eess
	It then follows that
	\bess
	|V(s,x_1)|&\le&v_1(t_{x_1}',x_1)+\int_{t_{x_1}'}^s
	\left|\frac{bu_1w_1}{1+w_1}-\frac{bu_2w_2}{1+w_2}+c(v_2-v_1)\rr|_{x=x_1}{\rm d}\tau\\
	&\le& L_v\|g,\,h\|_{C([0,T])}+
	TC_8\big(\|V(\cdot,x_1)\|_{C[t'_{x_1},t])}+\|U,\,W\|_{L^{\infty}(\Pi_T)}\big),
	\eess
	where $C_8=\max\{c,\,bA_1,\,b\}$. It follows from that
	\[|V(t,x_1)|\le C_{10}\kk(\|g,\,h\|_{C^1([0,T])}+T\|U,\,W\|_{L^{\infty}(\Pi_T)}\rr)\]
	if $T>0$ is sufficiently small.
	
\sk{\it Case 3:} $x_1\in (g_2(t),h_2(t))$ and $x_1\in [-h_0,h_0]$. In this case we
	can derive
	\[|V(t,x_1)|\le C_{11}T\|U,\, W\|_{L^{\infty}(\Pi_T)}\]
	by using similar methods. Since it is actually much simpler, we omit the details.
	In conclusion, we have proved the estimate \eqref{x1}.
\end{proof}

\begin{theorem}{\rm(Global solution)}\label{th2.2} The problem \eqref{1.1} has a
	unique global solution $(u,v,w,g,h)$, and
	there exist four positive constants $A_i$, $i=1,2,3,4$, such that
	\bess &(u,v,w,g,h)\in C^{1,1-}(\Pi_{\infty})\times C^{1,1-}(\Pi_{\infty})\times
	W^{1,2}_p(D^{\infty}_{g,h})\times
	[C^{1+\frac{\alpha}2}([0,\infty])]^2,\\[1mm]
	&0<u\le A_1 \ {\rm in} \ \Pi_{\infty};\ \ 0<v\le A_2, \ 0<w\le A_3 \ {\rm in} \
	D^{\infty}_{g,h};\ \ 0<-g'(t),h'(t)\le A_4 \ {\rm in} \ [0,\infty),
	\eess
	where $A_1=\max\left\{\|u_{0}\|_\infty,\,\theta/a\right\}$.
\end{theorem}

\begin{proof} It follows from Theorem \ref{th2.1} that the problem \eqref{1.1} has
	a unique local solution $(u,v,w,g,h)$ for
	some $0<T\ll1$ and $g'(t)<0$, $h'(t)>0$ for $0\le t\le T$.
	
	It is easy to show that $0<u\le A_1$ in $\Pi_{T}$. Recalling the equations of $(v,w)$ we can readily conclude that there exists $A_2,\ A_3>0$ such that $0<v\le A_2$, $0<w\le A_3$ in $D^T_{g,h}$.
	Making use of the similar arguments in the proof of \cite[Lemma 2.1]{WZjdde17}, we
	can show that there exists constant $A_4>0$ , which only depends on the initial data, such that $0<-g'(t),\,h'(t)\le A_4$  in $[0,T]$.
	
	With above estimates, we can extend the local solution uniquely to the
	global solution, and
	\bess
	(u,v,w,g,h)\in C^{1,1-}(\Pi_{\infty})\times C^{1,1-}(\ol
	D^{\infty}_{g,h})\times W^{1,2}_p(D^{\infty}_{g,h})\times
	[C^{1+\frac{\alpha}2}([0,\infty))]^2;
	\eess
	see \cite[Corollary 1.1]{Wdcds19} for the details. It follows from the standard parabolic
	regularity theory that $(u,v,w,g,h)$ is the
	unique classical solution of \eqref{1.1}. Combining $v(t,x)=0$ for
	$x\notin(g(t),h(t))$ and the equation satisfied by $v$, we easily derive that $v\in C^{1,1-}(\Pi_{\infty})$. The proof is ended.
\end{proof}

Since $g'(t)<0$, $h'(t)>0$, there exist $g_{\yy}\in[-\yy,0)$ and
$h_{\yy}\in(0,\yy]$ such that
\[\lim_{t\to\infty}g(t)=g_\infty, \ \ \ \lim_{t\to\infty}h(t)=h_\infty.\]
The case $h_\infty=-g_\infty=\infty$ is called {\it Spreading}, and the case $h_\infty-g_\infty<\infty$ is called {\it Vanishing}.

\begin{theorem}{\rm(Uniform estimates)}\label{th2.3}	Let $(u,v,w,g,h)$ be the unique global solution of \eqref{1.1}. Then there exists a constant $C>0$ such that
	\bes
	\|w(t,\cdot)\|_{C^1([g(t),h(t)])}\le C, \ \ \|g',\,h'\|_{C^{{\alpha}/2}([1,\yy))}\leq C, \ \ \forall\ t\ge 1.
	\label{x.a}\ees
\end{theorem}

\begin{proof} Remember $0\le v\le A_2$, $0\le w\le A_3$. The estimates \eqref{x.a} can be proved by using analogous methods in \cite[Theorem 2.1]{Wjfa16} for the case $h_\yy-g_\yy<\yy$ and \cite[Theorem 2.2]{WZjde18} for the case $h_\yy-g_\yy=\yy$. We omit the details here.
\end{proof}

\section{Preliminaries}

In this section, we will show some preliminaries which are crucial in the later
parts. First we will investigate an eigenvalue
problem and analyze the properties of its principal eigenvalue which will pave the ground for later discussion.	It is well known that the eigenvalue problem
\bess
\left\{\begin{aligned}
	&d\hat \phi_{xx}+a_{11}\hat\phi=\rho \hat\phi, &&l_1<x<l_2,\\
	&\hat\phi(l_i)=0,&&i=1,2
\end{aligned}\right.
\eess
has a principle eigenpair $(\rho_1,\hat\phi_1)$, where
\[\rho_1=a_{11}-\df{d\pi^2}{(l_2-l_1)^2},~ ~ ~\hat\phi_1(x)=\cos
\df{\pi(2x-l_2-l_1)}{2(l_2-l_1)}.\]

Now we consider the following eigenvalue problem
\bes
\left\{\begin{aligned}
	&d\phi_{xx}+a_{11}\phi+a_{12}\psi =\lambda \phi, &&l_1<x<l_2,\\
	&a_{21}\phi+a_{22}\psi =\lambda \psi, &&l_1<x<l_2,\\
	&\phi(l_i)=0,&&i=1,2
\end{aligned}\right.\label{3.1}
\ees
with $a_{12}$, $a_{21}>0$ and $a_{11}$, $a_{22}<0$. It is clear that if $(\lambda,(\phi,\psi))$ is an eigenpair of \eqref{3.1}, then $\lambda\not=a_{22}$ and $\psi(l_i)=0$, $i=1,2$. Define
\[\mathcal{L}=\begin{pmatrix}
d\partial_{xx}+a_{11}  && a_{12} \\
a_{21}  && a_{22}
\end{pmatrix},\]
and choose the domain of $\mathcal{L}$:
\[\mathcal{D}\big(\mathcal{L}\big)=\left\{(\phi,\psi)\in H^2((l_1,l_2))\times
L^2((l_1,l_2)):\phi(l_i)=0, \ i=1,2\right\}.\]
Similarly to the proof of \cite[Theorem 3.1]{WC-NA15}, we can prove the following
results by means of \cite[Theorem 2.3, Remark 2.2]{WZh}. The details are omitted here.

\begin{theorem}\label{th3.1} Let $\sigma(\mathcal{L})$ be the spectral set of
	$\mathcal{L}$ and
	$\mathfrak{s}(\mathcal{L}):=\sup\{Re \lambda: \lambda\in\sigma(\mathcal{L})\}$.
	Then the following properties hold true:
	
	\sk{\rm(i)}\, $\mathfrak{s}(\mathcal{L})$ is the principal eigenvalue of \eqref{3.1} with positive eigenvectors
	$(\phi_1,\psi_1)$;
	
	\sk{\rm(ii)}\,
	$\mathfrak{s}(\mathcal{L})=\df{1}{2}\big[\rho_1+a_{22}+\sqrt{(\rho_1-a_{22})^2+4a_{12}a_{21}}\big]$
	and has the
	same sign with $\rho_1-a_{12}a_{21}/{a_{22}}$;
	
	\sk{\rm(iii)}\, $\mathfrak{s}(\mathcal{L})$ is strictly monotone increasing in the
	length of the interval $(l_1,l_2)$ and strictly
	monotone decreasing in $d$.
\end{theorem}
By Theorem \ref{th3.1}, we can easily deduce the following result.
\begin{corollary}\label{cor3.1}Define
	$\Gamma=a_{11}-\frac{a_{12}a_{21}}{a_{22}}$. Let $\lambda_1$ be the principal
	eigenvalue of the problem \eqref{3.1}. Then the following properties are valid:
	
	\sk{\rm(i)}\, If $\Gamma\le0$, then $\lambda_1<0$ for any $d>0$ and $(l_1,l_2)$;
	
	\sk{\rm(ii)}\, If $\Gamma>0$, we fix the domain $(l_1,l_2)$ and let
	$d^*(l_1,l_2)=\Gamma(l_2-l_1)^2\pi^{-2}$. Then
	$\lambda_1>0$ when $0<d<d^*(l_1,l_2)$, $\lambda_1=0$ when  $d=d^*(l_1,l_2)$, and
	$\lambda_1<0$ when $d>d^*(l_1,l_2)$;
	
	\sk{\rm(iii)}\, If $\Gamma>0$, we fix $d>0$ and set $L^*(d)=\pi\sqrt{d/\Gamma}$.
	Then
	$\lambda_1>0$ when $l_2-l_1>L^*(d)$, $\lambda_1=0$ when $l_2-l_1=L^*(d)$ and
	$\lambda_1<0$ when  $l_2-l_1<L^*(d)$.
\end{corollary}
Let $\lambda_1$ be the principal eigenvalue of \eqref{3.1}, that is, two components of the corresponding eigenfunction are both positive or negative. Then we have
\[\lambda_1>a_{22}, \ \rho_1=\lambda_1-{a_{12}a_{21}}/({\lambda_1-a_{22}}).\]
Thus by the uniqueness of $\rho_1$ we easily derive the uniqueness of the principal eigenvalue of \eqref{3.1}.

Let $(\mu_1,u_1)$ be the first eigenpair of $-\Delta$ with homogeneous Dirichlet
boundary condition on $(l_1,l_2)$ and ($\lambda_1,(\phi_1,\psi_1))$ be the principal eigenpair of the problem \eqref{3.1}. The direct calculation yields
\bes\begin{pmatrix}
	-d\mu_1+a_{11}  && a_{12} \\
	a_{12}  && a_{22}
\end{pmatrix}
\begin{pmatrix}
	\langle\phi_1,u_1\rangle \\
	\langle\psi_1,u_1\rangle
\end{pmatrix}=\lambda_1
\begin{pmatrix}
	\langle\phi_1,u_1\rangle \\
	\langle\psi_1,u_1\rangle
\end{pmatrix},
\label{bb}\ees
where $\langle\cdot,\cdot\rangle$ denotes the inner product in $L^2((l_1,l_2))$.

The following lemma will play an important role in the study of long time
behaviors of $(u, v,w)$ when $h_\infty-g_\infty<\infty$.

\begin{lemma}\label{l3.2}\, Let $m(t,x)$ be a bounded function, $d$, $C$,
	$\mu$ and $\eta_0$ be positive constants, and constant $x_0<\eta_0$. Let
	$\eta\in C^1([0,\infty))$, $w\in W^{1,2}_p(D_T)$ and $w_0\in
	W^2_p((x_0,\eta_0))$ for any
	$T>0$ with some $p>1$, and $w_x\in C(D)$, where
	$D_T=\{(t,x):0<t<T, x_0<x<\eta(t)\}$,
	$D=\{(t,x):0\le t\yy,x_0<x\le \eta(t)\}$.
	Assume that $(w,\eta)$ satisfies
	\bess\left\{\begin{array}{lll}
		w_t-d w_{xx}+m(t,x)w_x\geq -C w, \ \ &t>0,\ \ x_0<x<\eta(t),\\[.5mm]
		w\ge 0,\ \ \ &t>0, \ \ x=x_0,\\[.5mm]
		w=0,\ \ \eta'(t)\geq-\mu w_x, \ &t>0,\ \ x=\eta(t),\\[.5mm]
		w(0,x)=w_0(x)\ge,\,\not\equiv 0,\ \ &x\in (x_0,\eta_0),\\[.5mm]
		\eta(0)=\eta_0,
	\end{array}\right.\eess
	and $\dd\lim_{t\to\infty} \eta(t)=\eta_\infty<\infty$,
	$\dd\lim_{t\to\infty} \eta'(t)=0$,
	\[\|w(t,\cdot)\|_{C^1([x_0,\,\eta(t)])}\leq M, \ \, \forall\, t\ge 1\]
	for some constant $M>0$. Then $\dd\lim_{t\to\infty}\,\max_{x_0\leq x\leq
		\eta(t)}w(t,x)=0$.
\end{lemma}
\begin{proof} When $x_0=0$ and $m(t,x)=0$, this lemma is exactly Proposition 2 in \cite{WZ-dcdsa18}; when $x_0=0$ and $m(t,x)=\gamma$ is a constant, this lemma is exactly Lemma 3.1 in \cite{ZWjmaa19}. For our present case, by the maximum principle we have $w(t,x)>0$ for $t>0$ and $x_0<x<\eta(t)$. If we follow the proof of \cite[Theorem 2.2]{WZjdde14} word by word we can prove this lemma. We will leave out the details because
	the advection term and boundary condition at $x =x_0$ do not influence the availability of the argument in \cite[Theorem 2.2]{WZjdde14}.
\end{proof}

\begin{lemma} $($Comparison principle$)$\label{l3.3} \,Let $T>0$, $\bar g,\bar h\in C^1([0,T])$ and $\bar g<\bar h$ in $[0,T]$. Let $\bar u\in C^{1,0}([0,T]\times\mathbb{R})$, $\bar v\in C^{1,0}(\overline O), \bar w\in C(\overline{O})\cap C^{1,2}(O)$ with $O=\{0<t\leq T,\,\bar g(t)<x<\bar h(t)\}$. Assume that $(\bar u, \bar v,\bar w,
	\bar g,\bar h)$ satisfies
	\bess
	\left\{\begin{aligned}
		&\bar u_{t}\ge \theta-a\bar u,& &t>0,~-\yy<x<\yy,\\
		&\bar v_t\ge f_2(\bar u, \bar v,\bar w),& &t>0,~\bar g(t)<x<\bar h(t),\\
		&\bar w_t-d\bar w_{xx}\ge f_3(\bar v,\bar w),& &t>0,~\bar g(t)<x<\bar h(t),\\
		&\bar v(t,x)=\bar w(t,x)=0,& &t>0, \ x=\bar g(t),\,\bar h(t),\\
		&\bar g'(t)\le-\mu \bar w_x(t,\bar g(t)),& &t\ge0,\\
		&\bar h'(t)\ge-\beta \bar w_x(t,\bar h(t)),& &t\ge0,\\
		&\bar v(0,x),\bar w(0,x)\geq 0, && \bar g(0)\le x\le\bar h(0).
	\end{aligned}\right.
	\eess
	If $\bar g(0)\leq-h_0,\,\bar h(0)\geq h_0$, $u_0(x)\leq\bar u(0,x)$ in $\mathbb{R}$, and $v_0(x)\leq\bar v(0,x), w_0(x)\leq\bar w(0,x)$ on $[-h_0,h_0]$. Then the solution $(u,v,w,g,h)$ of \eqref{1.1} satisfies
	\bess
	g\geq\bar g,\ h\leq\bar h \ \ {\rm on}\ \, [0,T]; \ \
	u\leq\bar u\ \  {\rm on}\ \, [0,T]\times\mathbb{R}; \ \
	v\leq\bar v, \ w\le\bar w\ \  {\rm on}\ \, \overline D_{g,h}^T,\eess
	where $D_{g,h}^T$ is defined as in the beginning of Section 2.
\end{lemma}

\begin{proof} Take $0<\rho<1$ and let $(u_\rho,v_\rho,w_\rho,g_\rho,h_\rho)$ be the corresponding unique solution of \eqref{1.1} with $(h_0, v_0, w_0)$ replaced by $(\rho h_0,v_{0,\rho}, w_{0,\rho})$, where $v_{0,\rho}(x), w_{0,\rho}(x)$ satisfy \eqref{1.2} with $h_0$ replaced by $\rho h_0$, and satisfy
	\bess
	0< v_{0,\rho}(x)\leq v_0(x),\ 0<w_{0,\rho}(x)\leq w_0(x)\ \ \, {\rm on}\ \ (-\rho h_0, \rho h_0),\eess
	as well as
	\bess\lim_{\rho\to 1}v_{0,\rho}\left(\rho x\right)= v_0(x) \ \ {\rm in}\ \ W_\yy^1((-h_0,h_0)),\ \ \ \lim_{\rho\to 1}w_{0,\rho}\left(\rho x\right)= w_0(x) \ \ \, {\rm in}\ \ W_p^2((-h_0,h_0)).\eess
	By a simple comparison consideration, we have $u_\rho\leq\bar u$ on $[0,T]\times\mathbb{R}$.
	Thus $(v_\rho,w_\rho)$ satisfies
	\bess
	\left\{\begin{aligned}
		&v_{\rho,t}\le f_2(\bar u, v_\rho, w_\rho),& &t>0,~g_\rho(t)<x<h_\rho(t),\\
		&w_{\rho,t}-d w_{\rho,xx}=f_3(v_\rho, w_\rho),& &t>0,~g_\rho(t)<x<h_\rho(t),\\
		&v_\rho(t,x)=w_\rho(t,x)=0,& &t>0, \ x\notin(g_\rho(t),h_\rho(t)),\\
		&g'_\rho(t)=-\mu w_{\rho,x}(t,g_\rho(t)),\,\,\,h'_\rho(t)=-\beta w_{\rho,x}(t,h_\rho(t)), &&t\ge0,\\[.5mm]
		&v_\rho(0,x)=v_{0,\rho}(x), \ \ \ w_\rho(0,x)=w_{0,\rho}(x), &&-\rho h_0\le x\le \rho h_0,\\[.5mm]
		& h_\rho(0)=-g_\rho(0)=\rho h_0.
	\end{aligned}\right.
	\eess
	Similarly to \cite[Lemma 3.5]{DL10}, by use of the indirect arguments and strong maximum principle we can show that $g_\rho(t)>\bar g(t), h_\rho(t)<\bar h(t)$ for $0\le t\le T$. Thus $v_\rho(t,x)<\bar v(t,x), w_\rho(t,x)<\bar w(t,x)$ for $0<t\le T$ and $g_\rho(t)\le x\le h_\rho(t)$ by the standard comparison principle. Letting $\rho\to 1$ and  using the continuous dependence of solutions on parameters we have $(u_\rho,v_\rho,w_\rho,g_\rho,h_\rho)\to (u,v,w,g,h)$. The details are omitted.
\end{proof}

\section{Long time behavior of $(u,v,w)$}

This section concerns with the long time behavior of $(u,v,w)$. We first study the vanishing case ($h_\infty-g_\infty<\infty$).

\begin{theorem}\label{th4.1}Let $(u,v,w,g,h)$ be the unique global solution of
	\eqref{1.1}. If $h_\infty-g_\infty<\infty$, then
	\bess
	&\dd\lim_{t\to\infty}\|w(t,\cdot)\|_{C([g(t),\,h(t)])}=\lim_{t\to\infty}\|v(t,\cdot)\|_{C([g(t),\,h(t)])}=0,&\\
	&\dd\lim_{t\to\infty}u(t,x)=\theta/a \ \ \mbox{ uniformly\, in } \ \mathbb{R}.&\eess
\end{theorem}

\begin{proof}  By the second estimate in \eqref{x.a}, we see that both $g'(t)$ and $h'(t)$ are uniformly continuous in $[1,\infty)$. Since $h_\infty-g_\infty<\infty$, it is easy to deduce $\dd\lim_{t\rightarrow\infty}g'(t)=\lim_{t\rightarrow\infty}h'(t)=0$. Then, using the first
	estimate of \eqref{x.a}, Lemma \ref{l3.2} in $[0,h(t))$ and a similar version of
	Lemma \ref{l3.2} in $(g(t),0]$, one can arrive at
	$\lim\limits_{t\to\infty}\|w(t,\cdot)\|_{C([g(t),\,h(t)])}=0$. For any $\ep>0$, there exists $T>0$ such that $bA_1w(t,x)/(1+w(t,x))\le \ep$ for $t\ge T$ and $x\in\mathbb{R}$. Thus $v$ satisfies
	\bess
	\left\{\begin{aligned}
		&v_{t}\le\ep-cv, &&t\ge T,~g(t)<x<h(t),\\
		&v(t,g(t))=v(t,h(t))=0, && t\ge T,\\
		&v(T,x)\ge0.
	\end{aligned}\right.
	\eess
	By the comparison principle, we have
	$\limsup\limits_{t\to\infty}\|v(t,\cdot)\|_{C([g(t),\,h(t)])}\le {\ep}/{c}$.
	The arbitrariness of $\ep$ implies
	$\lim\limits_{t\to\infty}\|v(t,\cdot)\|_{C([g(t),\,h(t)])}=0$.
	Similarly, we can easily deduce that
	\bes\dd\limsup_{t\to\infty}u(t,x)\le{\theta}/{a} \ \ {\rm uniformly ~ in~}\mathbb{R}.\label{4.1}\ees
	On the other hand, for any $\ep_1>0$, there exists $T_1>0$ such that $w(t,x)/(1+w(t,x))\le \ep_1$
	for $t\ge T_1$ and $x\in\mathbb{R}$. So $u$ satisfies
	\bess
	\left\{\begin{aligned}
		&u_{t}\ge\theta-(a+b\ep_1)u, &&t\ge T_1,~x\in\mathbb{R},\\
		&u(T_1,x)>0,&&x\in\R.
	\end{aligned}\right.
	\eess
	Let $\underline u$ be the unique solution of the problem
	\bess
	\left\{\begin{aligned}
		&\underline u_{t}=\theta-(a+b\ep_1)\underline u, &&t\ge T_1,\\
		&\underline u(T_1)=0.
	\end{aligned}\right.
	\eess
	By using comparison principle and the fact that $\dd\lim_{t\to\infty}\underline
	u(t)=\theta/(a+b\ep_1)$, we have that $\dd\liminf_{t\to\infty}u(t,\cdot)\ge\theta/(a+b\ep_1)$ uniformly in $\mathbb{R}$. Due to the arbitrariness of $\ep_1$ and \eqref{4.1}, we derive the desired result.
\end{proof}

In the following we study the spreading case ($h_\infty-g_\infty=\infty$).
To get the accurate limits of the solution components $u,v,w$ of \eqref{1.1}, we first give a proposition  which concerns the existence, uniqueness and asymptotic behavior of positive solution of a boundary value problem.

\begin{proposition}\label{pro3.1} Let $m,l$ be positive constants and consider the following problem
	\bes
	\left\{\begin{aligned}
		&f_2(m, v, w)=0, &&-l<x<l,\\
		&-dw_{xx}=f_3(v, w), &&-l<x<l,\\
		& w(x)=0,&&x=\pm l.
	\end{aligned}\right.\label{3.5}
	\ees
	
	\sk{\rm(i)}\, Let $\lambda_1$ be the principal eigenvalue of
	\bes
	\left\{\begin{aligned}
		&d\phi_{xx}-q\phi+k\psi =\lambda \phi, &&-l<x<l,\\
		&bm\phi-c\psi =\lambda \psi, &&-l<x<l,\\
		& \phi(\pm l)=0.
	\end{aligned}\right.\label{3.6}
	\ees
	Then \eqref{3.5} has a positive solution if and only if $\lambda_1>0$. Moreover, the positive solution of \eqref{3.5} is unique when it exists.
	
	\sk{\rm(ii)}\, To stress the dependence on $l$, we denote the unique positive solution of \eqref{3.5} by $(v_l,w_l)$. Then $(v_l,w_l)$ is nondecreasing in $l$, and converges to $(\hat v,\hat w)$ locally uniformly in $\mathbb{R}$ as $l\to\infty$, where $(\hat v,\hat w)$ is the unique positive root of
	\bes
	f_2(m, v, w)=0,\ \ \ f_3(v, w)=0.
	\label{3.7}
	\ees
\end{proposition}

\begin{proof}
	\sk{\rm(i)}\,
	Clearly, the problem \eqref{3.5} is equivalent to
	\bess
	\left\{\begin{aligned}
		&-dw_{xx}+qw=\frac{kbmw}{c(1+w)^2}, \ \ v=\frac{bmw}{c(1+w)}, \ \ &&-l<x<l,\\
		&w(x)=0,&&x=\pm l.
	\end{aligned}\right.
	\eess
	For clarity of exposition, we will always use the problem \eqref{3.5} in later discussion.
	
	If \eqref{3.5} has a positive solution $(v,w)$, it is easy to show that $q<\lambda_1(q)<kbm/c$, where $\lambda_1(q)$ is the principal eigenvalue of
	\bes
	\left\{\begin{aligned}
		&-d\phi_{xx}+q\phi=\lambda\phi, &&-l<x<l,\\
		&\phi(x)=0,&&x=\pm l.
	\end{aligned}\right.\label{3.8}
	\ees
	Moreover, since $\lambda_1(q)<{kbm}/{c}$, and the function ${bmk}/(c+x)-x$ is decreasing in $x>-c$, we can show that there exists the unique $\lambda^*>0$ such that ${bmk}/(c+\lambda^*)-\lambda^*=\lambda_1(q)$.
	Substituting this into \eqref{3.8}, one can easily see that $\lambda^*$ is the principal eigenvalue of \eqref{3.6}, that is, $\lambda^*=\lambda_1>0$.
	
	If $\lambda_1>0$, by the standard upper and lower solution methods we can show that \eqref{3.5} has at least one positive solution. Thanks to the structure of nonlinear terms of \eqref{3.5}, the uniqueness is easily derived.
	
	\sk{\rm(ii)}\, It follows from the above analysis and Corollary \ref{cor3.1} that for large $l$, \eqref{3.5} has a unique positive solution $(v_l,w_l)$ provided that $kbm>qc$. A comparison argument (Lemma 2.1 in \cite{DM01}) shows that $(v_l,w_l)$ is nondecreasing in $l$, and there exists $C>0$ such that $v_l,w_l<C$ for all large $l$. Making use of the standard elliptic regularity theory, we have that $(v_l,w_l)\to (\tilde v, \tilde w)$ in $C^2_{loc}(\mathbb{R})$, where $(\tilde v, \tilde w)$ is a positive solution of
	\bess
	\left\{\begin{aligned}
		&f_2(m, v, w)=0, &&-\infty<x<\infty,\\
		&-dw_{xx}=f_3(v, w), &&-\infty<x<\infty.
	\end{aligned}\right.
	\eess
	Obviously, $\tilde w$ satisfies
	\bes
	-d\tilde w_{xx}=\frac{kbm}{c(1+\tilde w)^2}\tilde w-q\tilde w, \ \ -\infty<x<\infty.
	\label{3.9}
	\ees
	Since $\frac{kbm}{c(1+w)^2}-q$ is decreasing in $w>0$, the possible positive solution of \eqref{3.9} is a unique positive root of $kbm=qc(1+w)^2$.
	Thus $\tilde w=\hat w$, and consequently $\tilde v=\hat v$. The proof is finished.
\end{proof}

\begin{theorem} Suppose that $h_\infty=-g_\infty=\infty$. If $\mathcal{R}_0+\sqrt{\mathcal{R}_0}>b/a$, then there are six positive constants $\underline{u}_\infty$, $\bar{u}_\infty$, $\underline{v}_\infty$, $\bar{v}_\infty$, $\underline{w}_\infty$ and $\bar{w}_\infty$ such that
	\bes\left\{\begin{array}{ll}
		\dd\underline{u}_\infty\le\liminf_{t\to\infty}u(t,x)\le\limsup_{t\to\infty}u(t,x)\le\bar{u}_\infty,\\
		\dd\underline{v}_\infty\le\liminf_{t\to\infty}v(t,x)\le\limsup_{t\to\infty}v(t,x)\le\bar{v}_\infty,\\
		\dd\underline{w}_\infty\le\liminf_{t\to\infty}w(t,x)\le\limsup_{t\to\infty}w(t,x)\le\bar{w}_\infty
	\end{array}\right.\label{x.e}\ees
	locally uniformly in $\mathbb{R}$. Particularly, if we assume $b\le 2a$, then
	\bes
	\lim_{t\to\yy}u(t,x)=u^*,\ \ \lim_{t\to\yy}v(t,x)=v^*,\ \ \lim_{t\to\yy}w(t,x)=w^* \ \ {\rm locally~uniformly~in~}\mathbb{R},\label{x.f}\ees
	where $(u^*,v^*,w^*)$ is a unique positive root of
	\bes
	f_1(u,w)=0,\ \  f_2(u,v,w)=0,\ \ f_3(v, w)=0.
	\label{3.10}
	\ees
\end{theorem}

\begin{proof} The condition $h_\infty-g_\infty=\infty$ implies $\mathcal{R}_0>1$ (cf. Theorem \ref{th5.1}). One can easily see that \eqref{3.10} has a unique positive root $(u^*,v^*,w^*)$. The following proof is actually an iterative process, and the idea comes from
	\cite{WZjdde17,WZhang16}.
	
\sk{\it Step 1:}  Clearly, \[\limsup_{t\to\infty}u(t,x)\le{\theta}/{a}=:\bar{u}_1 \ \mbox{ uniformly\, in } \ \mathbb{R}.\]
	Then for any $\ep>0$, there exists $T>0$ such that $u(t,x)\le {\theta}/{a}+\ep$ with $t\ge T$ and $x\in\mathbb{R}$. Thus $(v,w)$ satisfies
	\bess
	\left\{\begin{aligned}
		&v_t\le f_2(\theta/a+\ep,\,v, w), &&t>T,\ \ g(t)<x<h(t),\\
		&w_t-dw_{xx}=f_3(v, w), &&t>T,\ \ g(t)<x<h(t),\\
		&v(t,x)=0, \ w(t,x)=0,&&t>T,\ x=g(t) \ {\rm or} \ h(t),\\
		&v(T,x)\ge 0,\ w(T,x)\ge 0,&&g(T)\le x\le h(T).
	\end{aligned}\right.
	\eess
	Consider the ODEs problem
	\bes
	\left\{\begin{aligned}
		&\bar{v}_t=f_2(\bar{u}_1+\ep,\,\bar v, \bar w), \ \
		\bar{w}_t=f_3(\bar v, \bar w), &&t>T,\\
		&\bar v(T)=A_2, \ \bar w(T)=A_3.
	\end{aligned}\right.\label{3.11}
	\ees
	Since ${\cal R}_0>1$, the problem \eqref{3.11} has a unique positive equilibrium $(\bar{v}_1^{\ep},\bar{w}_1^{\ep})$ which is globally asymptotically stable. By a simple comparison consideration, we have $v(t,x)\le \bar{v}(t,x)$ and $w(t,x)\le \bar{w}(t,x)$ for $t\ge T$ and $x\in\mathbb{R}$. And so, $\dd\limsup_{t\to\infty}v(t,x)\le\bar{v}_1^{\ep}$ and $\dd\limsup_{t\to\infty}w(t,x)\le\bar{w}_1^{\ep}$ uniformly in $\mathbb{R}$. By the arbitrariness of $\ep$, we have
	\[\limsup_{t\to\infty}v(t,x)\le\bar{v}_1, \ \ \limsup_{t\to\infty}w(t,x)\le\bar{w}_1
	\ \mbox{ uniformly\, in } \ \mathbb{R},\]
	where $(\bar{v}_1,\bar{w}_1)$ is a unique positive root of the algebraic system  \eqref{3.7} with $m$ replaced by $\bar{u}_1$.
	
\sk{\it Step 2:}  For small $\ep>0$, there exists $T>0$ such that $w(t,x)\le \bar{w}_1+\ep$ for $t\ge T$ and $x\in\mathbb{R}$. Hence $u$ satisfies
	\bess
	\left\{\begin{aligned}
		&u_t\ge f_1(u, \bar{w}_1+\ep), &&t>T, \ x\in\mathbb{R},\\
		&u(T,x)>0, &&x\in\mathbb{R}.
	\end{aligned}\right.
	\eess
	Using the comparison argument with the solution having initial value $0$ we can deduce that
	\[\liminf_{t\to\infty}u(t,x)\ge \frac{\theta(1+\bar{w}_1)}{a+a\bar{w}_1+b\bar{w}_1}=:\underline{u}_1 \ \ {\rm uniformly~in} \ \ \mathbb{R}.\]
	Direct calculation shows that $\bar{w}_1=\sqrt{\mathcal{R}_0}-1$ and
	\[kb\underline{u}_1>qc \ {\rm ~if ~and ~only~ if~}\,
	\mathcal{R}_0+\sqrt{\mathcal{R}_0}>b/a.\]
	By our assumptions, we have $kb\underline{u}_1>qc$, and then  $kb(\underline{u}_1-\ep)>qc$ for small $\ep>0$.
	
	Recall Proposition \ref{pro3.1}. For any large $l$, let $(v_l,w_l)$ and $(\underline{v}_1^{\ep},\underline{w}_1^\ep)$ be the unique positive solutions of \eqref{3.5} and \eqref{3.7} with $m$ replaced by $\underline{u}_1-\ep$, respectively. Then $(v_l(x),w_l(x))\to (\underline{v}_1^{\ep}(x),\underline{w}_1^\ep(x))$ locally uniformly in $\mathbb{R}$ as $l\to\infty$. For any given $N\gg 1$ and $0<\sigma\ll 1$, there exists a large $l>N$ such that $v_l(x)>\underline{v}_1^{\ep}-\sigma/2$ and $w_l(x)>\underline{w}_1^\ep-\sigma/2$ for $x\in[-N,N]$.
	
	For such a fixed $l>N$, let $(\lambda_1, (\phi,\psi))$ be the principal eigenpair of \eqref{3.6} with $m$ replaced by $\underline{u}_1-\ep$. We can verify that for small $\delta>0$, $(\delta\psi,\delta\phi)$ is a lower solution of \eqref{3.5} with $m$ replaced by $\underline{u}_1-\ep$ (see the proof of Theorem \ref{th4.2} for details). Moreover, we may choose $T\gg 1$, $0<\delta\ll 1$ such that $[-l,l]\subseteq (g(t),h(t))$ for $t\ge T$, $u(t,x)\ge\underline{u}_1-\ep$ on $[T,\infty)\times[-l,l]$, and $\delta\psi(\cdot)\le v(T,\cdot), \delta\phi(\cdot)\le w(T,\cdot)$ on $[-l,l]$. Hence $(v,w)$ satisfies
	\bess
	\left\{\begin{aligned}
		&v_t\ge f_2(\underline{u}_1-\ep,\,v,w), &&t>T,\ \ -l<x<l,\\
		&w_t-dw_{xx}=f_3(v, w), &&t>T,\ \ -l<x<l,\\
		& w(t,x)>0,&&t>T,\ x=\pm l,\\
		&v(T,x)\ge\delta\psi(x),\ w(T,x)\ge\delta\phi(x),&&-l\le x\le l.
	\end{aligned}\right.
	\eess
	
	Let $(\tilde{v},\tilde{w})$ be a unique positive solution of the following problem
	\bess
	\left\{\begin{aligned}
		&\tilde {v}_t=f_2(\underline{u}_1-\ep,\,\tilde v, \tilde w), &&t>T,\ \ -l<x<l,\\
		&\tilde{w}_t-d\tilde{w}_{xx}=f_3(\tilde{v}, \tilde{w}), &&t>T,\ \ -l<x<l,\\
		& \tilde{w}(t,x)=0,&&t>T,\ x=\pm l,\\
		&\tilde{v}(T,x)=\delta\psi(x),\ \tilde{w}(T,x)=\delta\phi(x),&&-l\le x\le l.
	\end{aligned}\right.
	\eess
	Then $\tilde{v}$ and $\tilde{w}$ are nondecreasing in $t$. By the standard parabolic regularity we have $\dd\lim_{t\to\infty}(\tilde{v}(t,x),\tilde{w}(t,x))=(v_l(x),w_l(x))$ uniformly in $[-l,l]$. There exists $T_1>T$ such that $\tilde{v}(t,x)\ge v_l(x)-\sigma/2, \ \tilde{w}(t,x)\ge w_l(x)-\sigma/2$ for $t>T_1$ and $x\in[-l,l]$. Furthermore, by the comparison principle, $v(t,x)\ge \tilde{v}(t,x),  w(t,x)\ge\tilde{w}(t,x)$ for $t>T$ and $x\in[-l,l]$. So we have
	\[v(t,x)\ge\underline{v}_1^{\ep}-\sigma, \ \ w(t,x)\ge\underline{w}_1^{\ep}-\sigma \ \ \mbox{for} \ t>T_1,\ |x|\le N.\]
	These estimates combined with the arbitrariness of $\ep$, $\sigma$ and $N$ yield
	\[\liminf_{t\to\infty}v(t,x)\ge \underline{v}_1,\ \ \liminf_{t\to\infty}w(t,x)\ge \underline{w}_1 \ \ {\rm locally~uniformly~in} \ \ \mathbb{R},\]
	where $(\underline{v}_1,\underline{w}_1)$ is a unique positive root of \eqref{3.7} with $m$ replaced by $\underline{u}_1$.
	
\sk{\it Step 3:}  For any given $N>0$ and $0<\ep\ll 1$, there exists $T>0$ such that $w(t,x)\ge \underline{w}_1-\ep$ for $t>T$ and $-N\le x\le N$. So we have
	\bess
	\left\{\begin{aligned}
		&u_t\le f_1(u,\,\underline{w}_1-\ep), &&t>T, \ -N\le x\le N,\\
		&u(T,x)>0, &&-N\le x\le N.
	\end{aligned}\right.
	\eess
	Comparing with the following ODE problem
	\bess
	\bar{u}_t=f_1(\bar u,\,\underline{w}_1-\ep), \ \ t>T;\ \ \ \bar{u}(T)=A_1,
	\eess
	we can show that $u(t,x)\le\bar{u}(t)$ for $t\ge T$ and $-N\le x\le N$. Similarly to the preceding arguments, we have
	\[\limsup_{t\to\infty}u(t,x)\le \frac{\theta(1+\underline{w}_1)}{a+a\underline{w}_1+b\underline{w}_1}=:\bar{u}_2 \ \ {\rm locally~uniformly~in} \ \ \mathbb{R},\]
	and $\bar{u}_2>\underline{u}_1$. Moreover, the direct calculation yields
	$kb\bar{u}_2>qc$.
	
	\sk For any fixed $0<\ep\ll 1$, we take $K>\max\big\{A_3,\ \frac{kb(\bar{u}_2+\ep)}{qc}\big\}$ and consider the problem
	\bes
	\left\{\begin{aligned}
		&-dw_{xx}=\frac{kb(\bar{u}_2+\ep)w}{c(1+w)^2}-qw, &&-l<x<l,\\
		&w(\pm l)=K.
	\end{aligned}\right.\label{3.12}
	\ees
	Clearly, $kb(\bar{u}_2+\ep)>qc$. By the standard method we can show that \eqref{3.12} has a unique positive solution $w^l$ for large $l$. Moreover, $0<w^l\le K$. The comparison principle gives that $w^l$ is nonincreasing in $l$ and $w^l\ge w_l$. In the same way as the proof of Proposition \ref{pro3.1}\,(ii) we can derive
	$\dd\lim_{l\to\infty}w^l(x)=\bar{w}_2^\ep$ locally uniformly in $\mathbb{R}$,
	where $\bar{w}_2^\ep$ is a unique positive root of $kb(\bar{u}_2+\ep)=qc(1+w)^2$. Take
	\[v^l(x)=\frac{b(\bar{u}_2+\ep)w^l(x)}{c(1+w^l(x))}, \ \ \  \bar{v}_2^\ep=\frac{b(\bar{u}_2+\ep)\bar{w}_2^\ep}{c(1+\bar{w}_2^\ep)}.\]
	Then $\dd\lim_{l\to\infty}v^l(x)=\bar{v}_2^\ep$ locally uniformly in $\mathbb{R}$, and
	$(\bar{v}_2^\ep,\bar{w}_2^\ep)$ is a unique positive root of \eqref{3.7} in there $m$ is replaced by $\bar{u}_2+\ep$.
	
	For any given $N\gg 1$ and $0<\sigma\ll 1$, there exists a large $l>N$ such that $v^l(x)\le \bar{v}_2^\ep+\sigma$ and $w^l(x)\le \bar{w}_2^\ep+\sigma$ for $-N\le x\le N$. Moreover, there exists $T>0$ such that $u(t,x)\le \bar{u}_2+\ep$ for $(t,x)\in[T,\infty)\times[-l,l]$, and $h(T)>l, \ g(T)<-l$. Thanks to the equation of $v$ and $K>A_3$, we can find $T_1>T$ such that $v(t,x)\le \frac{b(\bar{u}_2+\ep)K}{c(1+K)}:=A_2^*$ on $[T_1,\infty)\times[-l,l]$. Therefore, $(v,w)$ satisfies
	\bess
	\left\{\begin{aligned}
		&v_t\le f_2(\bar{u}_2+\ep,\,v, w), &&t>T_1,\ \ -l<x<l,\\
		&w_t-dw_{xx}=f_3(v,w), &&t>T_1,\ \ -l<x<l,\\
		& w(t,x)\le K,&&t>T_1,\ x=\pm l,\\
		&v(T_1,x)\le A_2^*,\ w(T_1,x)\le K,&&-l\le x\le l.
	\end{aligned}\right.
	\eess
	Let $(\bar{v},\bar{w})$ be a unique positive solution of the problem
	\bess
	\left\{\begin{aligned}
		&\bar{v}_t=f_2(\bar{u}_2+\ep,\,\bar v, \bar w), &&t>T_1,\ \ -l<x<l,\\
		&\bar{w}_t-d\bar{w}_{xx}=f_3(\bar v, \bar w), &&t>T_1,\ \ -l<x<l,\\
		&\bar{w}(t,x)=K,&&t>T_1,\ x=\pm l,\\
		&\bar{v}(T_1,x)=A_2^*,\ \bar{w}(T_1,x)=K,&&-l\le x\le l.
	\end{aligned}\right.
	\eess
	Then we can deduce that $(\bar{v}(t,x),\bar{w}(t,x))\to(v^l(x),w^l(x))$ uniformly in $[-l,l]$ as $t\to\infty$. Thus there exists $T_2>T_1$ such that $\bar{v}(t,x)\le v^l(x)+\sigma, \bar{w}(t,x)\le w^l(x)+\sigma$ for $t>T_2$ and $x\in[-l,l]$. A comparison consideration yields that $v(t,x)\le \bar{v}(t,x)$ and $w(t,x)\le \bar{w}(t,x)$ on $[T_1,\infty]\times[-l,l]$. Recalling our previous conclusion we immediately derive that
	\[v(t,x)\le\bar{v}_2^{\ep}+2\sigma, \ \ w(t,x)\le\bar{w}_2^{\ep}+2\sigma, \ \ t>T_2, -N\le x\le N.\]
	The arbitrariness of $\ep$, $\sigma$ and $N$ implies
	\[\limsup_{t\to\infty}v(t,x)\le\bar{v}_2, \ \ \limsup_{t\to\infty}w(t,x)\le\bar{w}_2 \ \ \mbox{locally\, uniformly\, in }\ \mathbb{R},\]
	where $(\bar{v}_2,\bar{w}_2)$ is a unique positive root of the equations \eqref{3.7} with $m$ replaced by $\bar{u}_2$.
	
	We may argue as in Step 2 to conclude that
	\bess
	\liminf_{t\to\infty}u(t,x)\ge \underline{u}_2, \ \
	\liminf_{t\to\infty}v(t,x)\ge \underline{v}_2 , \ \liminf_{t\to\infty}w(t,x)\ge \underline{w}_2 \ \ {\rm locally~uniformly~in} \ \ \mathbb{R},\eess
	where $\underline{u}_2=\frac{\theta(1+\bar{w}_2)}{a+a\bar{w}_2+b\bar{w}_2}$, and $(\underline{v}_2,\underline{w}_2)$ is a unique positive root of \eqref{3.7} with $m$ replaced by $\underline{u}_2$.
	
\sk{\it Step 4:}  According to the above arguments we have
	\[\underline{u}_1<\underline{u}_2<\bar{u}_2<\bar{u}_1, \ \ \underline{v}_1<\underline{v}_2<\bar{v}_2<\bar{v}_1, \ \ \underline{w}_1<\underline{w}_2<\bar{w}_2<\bar{w}_1.\]
	Repeating the above procedures we can find six sequences $\{\underline{u}_n\}$, $\{\bar{u}_n\}$, $\{\underline{v}_n\}$, $\{\bar{v}_n\}$, $\{\underline{w}_n\}$ and $\{\bar{w}_n\}$ satisfying
	\bess
	&\underline{u}_1<\underline{u}_2<\cdots<\underline{u}_n<\cdots<\bar{u}_n<\cdots<\bar{u}_2<\bar{u}_1,\\
	&\underline{v}_1<\underline{v}_2<\cdots<\underline{v}_n<\cdots<\bar{v}_n<\cdots<\bar{v}_2<\bar{v}_1,\\
	&\underline{w}_1<\underline{w}_2<\cdots<\underline{w}_n<\cdots<\bar{w}_n<\cdots<\bar{w}_2<\bar{w}_1,
	\eess
	so that
	\bess
	&\dd\underline{u}_n\le\liminf_{t\to\infty}u(t,x)\le\limsup_{t\to\infty}u(t,x)\le\bar{u}_n,\\
	&\dd\underline{v}_n\le\liminf_{t\to\infty}v(t,x)\le\limsup_{t\to\infty}v(t,x)\le\bar{v}_n,\\
	&\dd\underline{w}_n\le\liminf_{t\to\infty}w(t,x)\le\limsup_{t\to\infty}w(t,x)\le\bar{w}_n
	\eess
	locally uniformly in $\mathbb{R}$.
	The limits of the above six sequences are well defined, and denoted by $\underline{u}_\infty$, $\bar{u}_\infty$, $\underline{v}_\infty$, $\bar{v}_\infty$, $\underline{w}_\infty$ and $\bar{w}_\infty$ respectively. It is clear that \eqref{x.e} holds.
	
	Now we assume $b\le 2a$ and prove \eqref{x.f}. By the careful calculations one can obtain
	\bess
	&\bar{u}_1=\dd\frac{\theta}{a},\ \ \frac{b\bar{u}_n\bar{w}_n}{1+\bar{w}_n}=c\bar{v}_n,\ \
	\frac{k\bar{v}_n}{1+\bar{w}_n}=q\bar{w}_n,
	\ \ \underline{u}_n=\frac{\theta(1+\bar{w}_n)}{a+a\bar{w}_n+b\bar{w}_n},&\\[.2mm]
	&\dd\frac{b\underline{u}_n\underline{w}_n}{1+\underline{w}_n}=c\underline{v}_n,\ \
	\frac{k\underline{v}_n}{1+\underline{w}_n}=q\underline{w}_n,
	\ \ \bar{u}_{n+1}=\frac{\theta(1+\underline{w}_n)}{a+a\underline{w}_n+b\underline{w}_n}.&
	\eess
	Consequently, $\underline{u}_\infty$, $\bar{u}_\infty$, $\underline{v}_\infty$, $\bar{v}_\infty$, $\underline{w}_\infty$, $\bar{w}_\infty$ satisfy
	\bess
	&\dd\frac{b\bar{u}_\infty\bar{w}_\yy}{1+\bar{w}_\yy}=c\bar{v}_\yy,\ \
	\frac{k\bar{v}_\yy}{1+\bar{w}_\yy}=q\bar{w}_\yy,
	\ \ \underline{u}_\yy=\frac{\theta(1+\bar{w}_\yy)}{a+a\bar{w}_\yy+b\bar{w}_\yy},&\\[.2mm]
	&\dd\frac{b\underline{u}_\yy\underline{w}_\yy}{1+\underline{w}_\yy}=c\underline{v}_\yy,\ \
	\frac{k\underline{v}_\yy}{1+\underline{w}_\yy}=q\underline{w}_\yy,
	\ \ \bar{u}_\yy=\frac{\theta(1+\underline{w}_\yy)}{a+a\underline{w}_\yy+b\underline{w}_\yy}.&
	\eess
	Using our assumptions $\mathcal{R}_0>1$ and $b/a\le 2$, we can derive after a series of calculations that
	\[\underline{u}_\infty=\bar{u}_\infty=u^*, \ \ \underline{v}_\infty=\bar{v}_\infty=v^*, \ \  \underline{w}_\infty=\bar{w}_\infty=w^*.\]
	Thus \eqref{x.f} holds and the proof is ended.
\end{proof}

\section{Criteria for spreading and vanishing}

In this section we study the criteria governing spreading ($h_\infty-g_\infty=\infty$) and vanishing ($h_\infty-g_\infty<\infty$).  In the following, we divide our discussion into two cases based on the {\it Basic Reproduction Number} ${\cal R}_0=\theta kb/(acq)$.
For convenience, we denote $\gamma=\max\left\{\mu,\beta\right\}$.

\subsection{The case ${\cal R}_0\le 1$}

\begin{theorem}\label{th5.1}Let $(u,v,w,g,h)$ be the unique solution of
	\eqref{1.1}. If ${\cal R}_0\le 1$, then
	$h_\infty-g_\infty<\infty$.
\end{theorem}

\begin{proof} By a simple comparison argument, we have
	\[u(t,x)\le {\theta}/{a}+\|u_0\|_\infty e^{-at}=:\hat u(t) \ \ {\rm for ~ }
	t\ge0,\ x\in\mathbb{R}.\]
	Hence $v$ satisfies
	\bess\left\{\begin{aligned}
		&v_{t}\le f_2(\hat u(t),\, v, w), &&t>0,~g(t)<x<h(t),\\
		&v(t,g(t))=v(t,h(t))=0, && t>0,\\
		&v(0,x)=v_0(x), &&|x|\le h_0.
	\end{aligned}\right.\eess
	Notice that ${\cal R}_0=\theta kb/acq\le 1$.	It follows that by simple calculations
	\bess \df{{\rm d} }{{\rm dt}}\int_{g(t)}^{h(t)}\big(cw+kv\big){\rm
		d}x&=&\int_{g(t)}^{h(t)}\big(cw_t+kv_t\big){\rm d}x\\
	&\le&\int_{g(t)}^{h(t)}\kk[cdw_{xx}+kb\|u_0\|_\infty e^{-at}w+(k\theta
	b/a-qc)w\rr]{\rm d}x\\
	&\le&\int_{g(t)}^{h(t)}\left(cdw_{xx}+kb\|u_0\|_\infty e^{-at}w\right){\rm d}x\\
	&\le&-cd\gamma^{-1}(h'(t)-g'(t))+kbA_3\|u_0\|_\infty e^{-at}(h(t)-g(t)).
	\eess
	Set
	\[f(t)=\int_{g(t)}^{h(t)}\big(cw+kv\big){\rm d}x,\ \ \ell(t)=h(t)-g(t),\ \
	\varphi(t)=kbA_3\|u_0\|_\infty e^{-at}.\]
	Then we have
	\[cd\ell'(t)\le-\gamma f'(t)+\gamma\varphi(t)\ell(t).\]
	Integrating the above differential inequality from $0$ to $t$ yields
	\[\ell(t)\le \ell(0)+{\gamma}(cd)^{-1}f(0)+{\gamma}(cd)^{-1}\int_{0}^{t}
	\varphi(s)\ell(s){\rm d}s.\]
	By virtue of the Gronwall inequality,
	\[\ell(t)\le\big[\ell(0)+{\gamma}(cd)^{-1}f(0)\big]\exp\left\{{{\gamma}(cd)^{-1}\int_{0}^{t}\varphi(s){\rm
			d}s}\right\}<\infty.\]
	Thus, $h_{\yy}-g_{\yy}<\yy$.\end{proof}

\subsection{The case ${\cal R}_0>1$}

In this subsection, we always assume that ${\cal R}_0>1$, and consider $d$,
$h_0$, $\mu$ and $\beta$ as varying parameters to depict the criteria for spreading and vanishing.

\begin{theorem}	\label{th4.2} Let $(u,v,w,g,h)$ be the solution of \eqref{1.1}. If $h_{\yy}-g_{\yy}<\yy$, then we have
	$$h_{\yy}-g_{\yy}\le \pi\sqrt{acd/(kb\theta-acq)}=:\Lambda.$$
	This implies that if $h_0\ge \Lambda/2$, then $h_{\yy}-g_{\yy}=\yy$. Moreover if
	$h_{\yy}-g_{\yy}=\yy$, then
	\bes\label{5.1}\dd\limsup_{t\to\yy}\|v(t,\cdot)\|_{C([g(t),h(t)])}>0, \ \
	\dd\limsup_{t\to\yy}\|w(t,\cdot)\|_{C([g(t),h(t)])}>0.
	\ees
\end{theorem}

\begin{proof} Due to Theorem \ref{th4.1} and $h_{\yy}-g_{\yy}<\yy$, we see that $\dd\lim_{t\to\infty}u(t,\cdot)=\theta/{a}$ uniformly in $\mathbb{R}$, and
	\bes\label{4.3}
	\lim_{t\to\yy}\|v(t,\cdot)\|_{C([g(t),h(t)])}=0,\ ~\ ~
	\lim_{t\to\yy}\|w(t,\cdot)\|_{C([g(t),h(t)])}=0.
	\ees
	Arguing indirectly,	if $h_\yy-g_\yy>\Lambda$, then there exists $T>0$ such that for
	any small $\ep>0$ satisfying $\frac{kb}{c}(\frac{\theta}{a}-\ep)-q>0$, we have
	\bess
	&u(t,x)>{\theta}/{a}-\ep,\ \ \forall \ t\ge T, \  x\in\mathbb{R};\\
	&h(t)-g(t)>\pi \sqrt{d}\kk[kb({\theta}/{a}-\ep)/c-q\right]^{-1/2}=:\Lambda_\ep,\
	\ \forall \ t\ge T.
	\eess
	Then for any $[l_1,l_2]\subseteq (g(T),h(T))$ and $l_2-l_1>\Lambda_\ep$, we have
	\bes
	\left\{\begin{aligned}
		&v_t\ge f_2({\theta}/{a}-\ep,\,v, w), &&t>T,\ \ l_1<x<l_2,\\
		&w_t-dw_{xx}=f_3(v, w), &&t>T,\ \ l_1<x<l_2,\\
		& w(t,x)>0,&&t>T,\ x=l_i,\ i=1,2,\\
		&v(T,x)>0,\ w(T,x)>0,&&l_1\le x\le l_2.
	\end{aligned}\right.
	\label{x.c}\ees
	
	Consider the following eigenvalue problem
	\bes
	\left\{\begin{aligned}
		&d\phi_{xx}-q\phi+k\psi =\lambda \phi, &&l_1<x<l_2,\\
		&b(\theta/{a}-\ep)\phi-c\psi =\lambda \psi, &&l_1<x<l_2,\\
		&\phi(l_i)=0,&&i=1,2.
	\end{aligned}\right.\label{3.3}
	\ees
	Denote the principal eigenpair of \eqref{3.3} by $(\lambda_1,(\phi,\psi))$  with $\max_{x\in[l_1,l_2]}|\phi(x)|=1$. It
	follows from Corollary \ref{cor3.1} that
	$\lambda_1>0$ due to $l_2-l_1>\Lambda_\ep$. Let
	\[\underline{v}(x)=\delta \psi(x), \ \ \ \underline{w}(x)=\delta \phi(x)\]
	with $\delta>0$ to be determined later.
	
	We claim that there exists $\delta_0>0$ sufficiently small such that for any $0<\delta<\delta_0$, we have
	\bes
	\left\{\begin{aligned}
		&0< f_2(\theta/a-\ep,\,\underline v, \underline w), \ \
		l_1<x<l_2,\\
		&-d\underline w_{xx}\le f_3(\underline v, \underline w),\ \
		l_1<x<l_2.
	\end{aligned}\right.\label{3.4}
	\ees
	In fact, we have
	\bess
	f_2(\theta/a-\ep,\,\underline v, \underline w)=
	\delta\frac{b(\theta/a-\ep)\phi}{1+\delta\phi}-c\delta\psi
	=\delta\frac{(c+\lambda_1)\psi}{1+\delta\phi}-c\delta\psi =\delta\left(\frac{c+\lambda_1}{1+\delta\phi}-c\right)\psi>0
	\eess
	provided that $\delta>0$ is small. The proof of the second inequality of \eqref{3.4} can
	be done in a similar manner.
	
	Furthermore, one can choose small $\delta>0$ such that $v(T,x)\ge \delta\psi(x)$ and $w(T,x)\ge \delta\phi(x)$ for $x\in[l_1,l_2]$.
	Then $(\underline{v},\underline{w})$ satisfies
	\bess
	\left\{\begin{aligned}
		&\underline v_t\le f_2(\theta/a-\ep,\,\underline v, \underline w), &&t>T,\ \ l_1<x<l_2,\\
		&\underline w_t-d\underline w_{xx}\le f_3(\underline v, \underline w), &&t>T,\ \ l_1<x<l_2,\\
		&\underline v(t,x)=0, \ \underline w(t,x)=0,&&t>T,\ x=l_i,\ i=1,2,\\
		&\underline v(T,x)\le v(T,x),\ \underline w(T,x)\le w(T,x),&&l_1\le x\le l_2.
	\end{aligned}\right.
	\eess
	By virtue of the comparison principle,
	\bes
	v(t,x)\ge \underline{v}(x),\ \ \ w(t,x)\ge \underline{w}(x),  \ \ t\ge T, \ l_1\le x\le
	l_2.\label{x.d}\ees
	This is a contradiction with \eqref{4.3}.
	
	We now assume $h_{\yy}-g_{\yy}=\yy$ and prove \eqref{5.1}. By the comparison principle, it is easy to see that $\lim\limits_{t\to\yy}\|v(t,\cdot)\|_{C([g(t),h(t)])}=0$ if and only if
	$\lim\limits_{t\to\yy}\|w(t,\cdot)\|_{C([g(t),h(t)])}=0$. Hence if we assume that one of the two limits in \eqref{5.1} does not hold,
	we can similarly obtain $\dd\lim_{t\to\infty}u(t,\cdot)=\theta/{a}$ uniformly in
	$\mathbb{R}$. By $h_{\yy}-g_{\yy}=\yy$, we can derive an analogous contradiction as above. The proof is ended.
\end{proof}

Obviously, $h_0\ge \Lambda/2$ is equivalent to $d\le 4h_0^2q({\cal
	R}_0-1)\pi^{-2}=:D$. So the above result suggests that
when ${\cal R}_0>1$, the larger initial habitat $[-h_0,h_0]$ or the lower dispersal
rate $d$ of the virus is, the more possibility of successful spreading is observed.

\begin{theorem}\label{th4.3} If $h_{\yy}-g_{\yy}=\yy$, then $h_{\yy}=\yy$ and $g_{\yy}=-\yy$.
\end{theorem}

\begin{proof} By way of contradiction, we assume that $h_{\yy}<\yy$ and $g_{\yy}=-\yy$. If we take $L>\Lambda+2$, where $\Lambda$ is defined in Theorem \ref{th4.2}, then there exists $T_0>0$ such that $g(T_0)<-L$. Then $w$ satisfies
	\bess
	\left\{\begin{aligned}
		&w_t-dw_{xx}=f_3(v, w), &&t>T_0,\ \ -L<x<h(t),\\
		&w(t,-L)>0,\ w(t,h(t))=0,&&t>T_0,\\
		&h'=-\beta w_x(t,h(t)),&&t>T_0,\\
		&w(T_0,x)\ge 0,&&-L\le x\le h(T_0).
	\end{aligned}\right.
	\eess
	As $h_\yy<\yy$, using the second estimate in \eqref{x.a} we have $\dd\lim_{t\rightarrow\infty}h'(t)=0$. Then, using the first estimate in \eqref{x.a} and Lemma \ref{l3.2}, one can arrive at
	\bes\lim\limits_{t\to\infty}\|w(t,\cdot)\|_{C([-L,\,h(t)])}=0.
	\label{x.b}\ees
	
	Then we may argue as in the proof of Theorem \ref{th4.1} with minor modifications to derive that
	\[\lim_{t\to\yy}\max_{[1-L,\,h(T_0)]} v(t,\cdot)=0,\ \ \lim_{t\to\yy}\max_{[1-L,\,h(T_0)]}u(t,\cdot)={\theta}/{a}.\]
	There exists $T>T_0$ such that $u(t,x)\ge {\theta}/{a}-\ep$ for $(t,x)\in[T,\yy)\times[1-L,\,h(T_0)]$. Let $\ep>0$ be small enough satisfying $L-1>\Lambda_\ep$, where $\Lambda_\ep$ is defined as in Theorem \ref{th4.2}. Choose an interval $[l_1,l_2]\subset(1-L,\,h(T_0))$ with $l_2-l_1>\Lambda_\ep$. Then $(v,w)$ satisfies \eqref{x.c}. As in the proof of Theorem \ref{th4.2}, we can conclude that \eqref{x.d} holds. This is a contradiction with \eqref{x.b}.
	
	Analogously, we can prove that the case with $h_{\yy}=\yy$ and $g_{\yy}>-\yy$ also does not hold. Therefore, we must have $h_{\yy}=\yy$ and $g_{\yy}=-\yy$.\end{proof}

The following result implies that although the initial habitat is small or the
dispersal rate is fast, the spreading also can occur if the expanding rate $\mu$ or $\beta$ is appropriately large. By using similar method in the proof of \cite[Lemma 3.2]{WZh15} with some modifications, we can prove the following lemma.

\begin{lemma}\label{l4.2} If $h_0< \Lambda/2$ $($or $d> D)$, then there exists
	$\mu^0>0$ $(resp.\ \beta^0>0)$ such that if
	$\mu\ge\mu^0$ $(resp.\ \beta\ge\beta^0)$, then $h_{\yy}-g_{\yy}=\yy$.
\end{lemma}

The above lemma also indicates that if $\gamma=\max\left\{\mu,\beta\right\}\ge\max\{\mu^0,\beta^0\}$, then
$h_{\yy}-g_{\yy}=\yy$. Instinctively, we deem that if ${\cal R}_0>1$, $h_0< \Lambda/2$ (or $d> D$), $\mu$ and $\beta$ both are small, then the
vanishing will happen. The lemma listed below  supports our belief.

\begin{lemma} Assume $h_0< \Lambda/2$ $($or $d> D)$. Then there exists $\mu_0>0$ such that when $\gamma=\max\left\{\mu,\beta\right\}\le\mu_0$, we must have
	$h_{\yy}-g_{\yy}<\yy$.
\end{lemma}
\begin{proof}Let $\hat u$ be the unique solution of the problem
	\bess
	\hat u_{t}=\theta-a\hat u,\ \ t>0;\ \ \
	\hat u(0)=\max\{\|u_0\|_{\yy},\,\theta/a\}.
	\eess
	Then $\hat u(t)\ge\theta/a$ and $\dd\lim_{t\to\yy}\hat u(t)=\theta/a$. By the comparison principle, $u(t,x)\leq\hat u(t)$ in $[0,\yy)\times\mathbb{R}$. For any fixed $h_0<l<\Lambda/2$, we consider the following eigenvalue problem
	\bess
	\left\{\begin{aligned}
		&d\phi_{xx}-q\phi+k\psi =\lambda \phi, &&-l<x<l,\\
		&(b\theta/a)\phi-c\psi =\lambda \psi, &&-l<x<l,\\
		&\phi(\pm l)=0.
	\end{aligned}\right.
	\eess
	In view of Corollary \ref{cor3.1}, the principal eigenvalue $\lambda_1<0$ since $2l<\Lambda$. Moreover, by \eqref{bb}, there exists a positive constant $\tilde{\phi}$ such that
	\bes
	-d\df{\pi^2}{(2l)^2}\tilde{\phi}-q\tilde{\phi}+k =\lambda_1
	\tilde{\phi},\ \ \
	\frac{b\theta}a\tilde{\phi}-c =\lambda_1.
	\label{yy}\ees
	
	Define
	\bess
	&f(t)=M\exp\left\{\dd\int_{0}^{t}\left[\tilde{\phi} b(\hat
	u(s)-{\theta}/{a})+\lambda_1\right]{\rm d}s\right\},\\
	&r(t)=\left(h^2_0+\gamma\pi\tilde{\phi}\dd\int_{0}^{t}f(s){\rm d}s\right)^{1/2},\\
	&\hat v(t,x)=f(t)\cos\df{\pi x}{2r(t)},\ \
	\hat w(t,x)=\tilde{\phi} f(t)\cos\df{\pi x}{2r(t)}, \ \ -r(t)\le x\le r(t),
	\eess
	where $\gamma=\max\left\{\mu,\beta\right\}$, and $M>0$ is taken large so that
	\bess
	v_0(x)\le M\cos\df{\pi x}{2h_0}=\hat v(t,0),\ \ w_0(x)\le \tilde{\phi}
	M\cos\df{\pi x}{2h_0}=\hat w(t,0)\ \ \mbox{in}\ \ [-h_0,h_0].
	\eess
	As $\dd\lim_{t\to\yy}\hat u(t)=\theta/a$ and $\lambda_1<0$, it follows that
	$\int_{0}^{\yy}f(s){\rm d}s<\yy$. Clearly, $r'(t)>0$ for $t\ge0$. Set
	\[\mu_0=\df{l^2-h^2_0}{\pi\tilde\phi\int_{0}^{\yy}f(s){\rm d}s}.\]
	Then $r(t)<l$ for $t\ge0$ provided $0<\gamma\le \mu_0$.
	
	Using \eqref{yy} and $\hat u(t)\ge\theta/a$, $r(t)<l$ for all $t\ge0$, by a series of calculations we have
	\bess
	\hat v_t-f_2(\hat u, \hat{v}, \hat{w})&\ge& f(t)\cos\df{\pi
		x}{2r(t)}\kk(\lambda_1+c-\df{\theta}{a}b\tilde{\phi}\rr)=0,\\
	\hat{w}_t-d\hat{w}_{xx}-f_3(\hat{v}, \hat{w})&\ge&f(t)\cos\df{\pi
		x}{2r(t)}\left[\tilde{\phi}^2b\kk(\hat{u}-\frac{\theta}{a}\rr)
	+d\tilde{\phi}\kk(\df{\pi}{2r(t)}\rr)^2+\lambda_1\tilde{\phi}+q\tilde{\phi}-k\right]\\
	&=&\tilde{\phi}f(t)\cos\df{\pi
		x}{2r(t)}\left[\tilde{\phi}b\kk(\hat{u}-\frac{\theta}{a}\rr)+d\kk(\df{\pi}{2r(t)}\rr)^2
	-d\kk(\df{\pi}{2l}\rr)^2\right]\ge0
	\eess
	for $t>0$ and $-r(t)< x< r(t)$. And we easily see
	\[-r'(t)=-\gamma\hat{w}_x(t,-r(t)),\ \ r'(t)=-\gamma\hat{w}_x(t,r(t)).\]
	Thus for any $0<\gamma\le\mu_0$, $(\hat{u},\hat{v},\hat{w},-r,r)$ satisfies $r(0)=h_0$ and
	\bess
	\left\{\begin{aligned}
		&\hat{u}_t=\theta-a\hat{u},&&t>0,\\
		&\hat v_t\ge f_2(\hat u, \hat{v}, \hat{w}),& &t>0,~-r(t)<x<r(t),\\
		&\hat w_t-d\hat w_{xx}\ge f_3(\hat{v}, \hat{w}),& &t>0,~-r(t)<x<r(t),\\
		&\hat v(t,\pm r(t))=\hat w(t,\pm r(t))=0,& &t>0,\\
		&-r'(t)\le-\mu \hat w_x(t,-r(t)), \
		r'(t)\ge-\beta \hat w_x(t,r(t)),& &t\ge0,\\
		&\hat{u}(0)\ge u_0(x),&&-\infty<x<\infty,\\
		&\hat v(0,x)\ge v_0(x),~ \hat w(0,x)\ge w_0(x),& &|x|\le h_0.
	\end{aligned}\right.
	\eess
	By the comparison principle (Lemma \ref{l3.3}), $-r(t)\le g(t)$, $h(t)\le r(t)$ for $t\ge 0$. As a result, we have
	$$	g_{\yy}\ge-\lim_{t\to\yy}r(t)\ge-l,\ ~\ ~h_{\yy}\le\lim_{t\to\yy}r(t)\le l,$$
	which implies $h_{\yy}-g_{\yy}<\yy$. This completes the proof.
\end{proof}

According to the above proof, we can see that $\mu_0$ is independent of $v_0$ and
$w_0$ and strictly decreasing in $M$.
Thus for any given $\mu$ and $\beta$, there exists $M>0$ sufficiently small such
that $\gamma\le \mu_0$. Meanwhile, if both $v_0$ and $w_0$ are small enough such that for such
$M$
\bess
v_0(x)\le M\cos\df{\pi x}{2h_0},\ \ w_0(x)\le \tilde{\phi} M\cos\df{\pi
	x}{2h_0},\ \ \forall \ x\in[-h_0,h_0],
\eess
we still can derive $h_{\yy}-g_{\yy}<\yy$ by the above arguments. Hence we
have the following conclusion.
\begin{remark} Assume $h_0< \Lambda/2 \, (d> D)$, and that $(v_0,w_0)$
	satisfies \eqref{1.2}. Then vanishing happens if both $v_0$
	and $w_0$ are small enough.
\end{remark}
Combining the above two lemmas, by the similar arguments in \cite[Theorem
5.2]{WZjdde17} we can show the following conclusion.
\begin{theorem}If $h_0< \Lambda/2 \, (d> D)$. There exists $0<\mu_*\le\mu^*$
	such that $h_{\yy}-g_{\yy}<\yy$ if
	$\gamma\le\mu_*$ or $\gamma=\mu^*$, and $h_{\yy}-g_{\yy}=\yy$ if $\gamma>\mu^*$.
\end{theorem}

\section{Discussion}

In this paper we proposed a viral propagation model with nonlinear infection rate and free boundaries and investigated the dynamical properties. This model is composed of two ordinary differential equations and one partial differential equation, in which the  spatial range of the first equation is the whole space $\R$, and the last two equations have free boundaries. As a new mathematical model, we have proved the existence, uniqueness and uniform estimates of global solution, and provided the criteria for spreading and vanishing, and long time behavior of the solution components $u,v,w$.

Comparing with the corresponding ordinary differential systems, the {\it Basic Reproduction Number} ${\cal R}_0=\theta kb/(acq)$ plays a different role:

\begin{itemize}
	\item[(i)] For the corresponding ordinary differential systems, by the Lyapunov function method we can prove that if ${\cal R}_0<1$ then the infection can not spread successfully, while if ${\cal R}_0>1$ then the infection will spread successfully. When ${\cal R}_0=1$ the dynamical property is not clear;
	
	\item[(ii)] For our present model, the results indicate that when ${\cal R}_0\le 1$, the virus cannot spread successfully; when ${\cal R}_0>1$, the successful spread of virus depends on the initial value and varying parameters. If the initial occupying
	area $[-h_0,h_0]$ is beyond a critical size, namely $2h_0\ge
	\pi\sqrt{acd/(bk\theta-acq)}$, then {\it spreading} happens
	regardless of the moving parameter $\mu$, $\beta$ and initial population density
	$(u_0,v_0,w_0)$. While $2h_0<
	\pi\sqrt{acd/(bk\theta-acq)}$, whether {\it spreading} or {\it vanishing} happens
	depends on the initial population density
	$(v_0,w_0)$ and the moving parameter $\mu$ and $\beta$.
\end{itemize}
From a biological point of view, our model and results seem closer to the reality. On the other hand, our model shows more complex and precise dynamical properties from a mathematical point of view.

Before ending this paper, we mention that for the corresponding Cauchy problem:
\bess\begin{cases}
	u_t-d_1\Delta u=f_1(u,w),  &x\in\mathbb{R}^n, \ t>0,\\
	v_t-d_2\Delta v=f_2(u,v,w), &x\in\mathbb{R}^n, \ t>0,\\
	w_t-d_3\Delta w=f_3(v,w), &x\in\mathbb{R}^n, \ t>0,\\
	u(0,x)=u_0(x), \ v(0,x)=v_0(x),\ w(0,x)=w_0(x), &x\in\mathbb{R}^n,
\end{cases}\eess
we guess that the {\it Basic Reproduction Number} is still ${\cal R}_0=\theta kb/(acq)$ and it plays
the same role as in the corresponding ODEs.

\vskip 4pt
{\bf Acknowledgments}: The authors would like to thank the reviewers for their helpful
comments and suggestions that significantly improve the initial version of this paper.

\end{document}